\documentclass[11pt]{amsart}
\usepackage{amssymb,amsmath,mathrsfs,enumerate,xparse,mathtools}
\usepackage{xspace}
\usepackage[pagebackref, colorlinks=true,linkcolor=blue,citecolor=blue]{hyperref}

\usepackage{graphicx}
\usepackage{fancyhdr}
\usepackage[margin=2.5cm]{geometry}
\hypersetup{
 colorlinks   = true,
 urlcolor     = blue,
 linkcolor    = blue,
 citecolor   = red ,
 bookmarksopen=true
}

\theoremstyle{plain}
\newtheorem{thrm}{Theorem}[section]
\newtheorem{lemma}[thrm]{Lemma}
\newtheorem{prop}[thrm]{Proposition}
\newtheorem{cor}[thrm]{Corollary}
\newtheorem{rmrk}[thrm]{Remark}
\newtheorem{dfn}[thrm]{Definition}
\newtheorem{prob}[thrm]{Problem}
\newtheorem{hyp}[thrm]{Hypothesis}
\newcommand{\Hsolv}{Hypothesis~\ref{H:solv}\xspace}
\newcommand{\Hcomp}{Hypothesis~\ref{H:comp}\xspace}
\newcommand{\Hcut}{Hypothesis~\ref{H:cut}\xspace}
\newcommand{\Hgrow}{Hypothesis~\ref{H:grow}\xspace}

\allowdisplaybreaks
\begin{document}
\newcommand{\Rn}{\mathbb R^n}
\newcommand{\R}{\mathbb R}
\newcommand{\n}{\nabla}
\newcommand{\ve}{\varepsilon}
\newcommand{\sa}{\langle}
\newcommand{\da}{\rangle}
\newcommand{\bM}{\mathbb M}
\newcommand{\Ric}{\operatorname{Ric}}

\numberwithin{equation}{section}

\title[]{A nonlinear Li-Yau inequality and its consequences}

\author{Agnid Banerjee}
\address{School of Mathematical and Statistical Sciences\\ Arizona State University}\email[Agnid Banerjee]{agnid.banerjee@asu.edu}

\author{Nicola Garofalo}
\address{School of Mathematical and Statistical Sciences\\ Arizona State University}\email[Nicola Garofalo]{nicola.garofalo@asu.edu}

\author{Hamidreza Mahmoudian}
\address{School of Mathematical and Statistical Sciences\\ Arizona State University}\email[Hamidreza Mahmoudian]{hamidreza.mahmoudian@asu.edu}

\begin{abstract}
We prove a sharp nonlinear analogue of the Li-Yau inequality for positive solutions of the normalized parabolic $p$-Laplacian equation
\[
u_t = \Delta u + (p-2)|\nabla u|^{-2}\nabla^2u(\nabla u,\nabla u),\ \ \  1<p<\infty,
\]
on a closed Riemannian manifold with nonnegative Ricci curvature, the equation being interpreted in the viscosity sense on the critical set of the solution. We also derive a sharp global
Harnack inequality from the gradient estimate.

The constant is best possible. Equality holds identically for an explicit self-similar solution in $\Rn$, and sharpness on closed manifolds follows by taking the limit along the flat tori $\mathbb R^n/(L\mathbb Z)^n$ as $L \to \infty$. The proof rests on an exact Bochner-type identity
for a family of uniformly parabolic approximating flows.
We choose the first-order term separately from the
second-order regularization so that the maximum principle
yields the sharp constant, uniformly in the regularization
parameter and without any assumption on the critical set
of the solution.

Using this Bochner identity with a modified Li-Yau
functional, we extend the estimate to closed manifolds with
$\operatorname{Ric} \ge -\kappa$. 
\end{abstract}

%
%

\maketitle

\tableofcontents

\section{Introduction}\label{S:intro}

Let $\bM$ be a closed (compact, without boundary) Riemannian manifold of dimension $n$ satisfying $\operatorname{Ric}(\bM)\ge 0$. We consider positive solutions of the normalized parabolic $p$-Laplacian equation

\begin{equation}\label{0}
u_t = \Delta u + (p-2)\, |\nabla u|^{-2}\, \nabla^2u(\nabla u,\nabla u), \ \ \ \ \ \ 1<p<\infty,
\end{equation}
which is understood in the viscosity sense: at critical points of $u$ the operator is interpreted by means of its upper and lower semicontinuous envelopes, see Section \ref{S:critical} below and \cite{Do, BG, JS, LTW}. Although the operator in \eqref{0} is not in divergence form, and is discontinuous at critical points, it is \emph{uniformly elliptic}, in the sense that, when expressed in the form
\[
u_t = \sum_{i,j=1}^n a_{ij}(\nabla u)u_{ij},
\]
its coefficient matrix
\[
a_{ij}(\sigma) = \delta_{ij} + (p-2) \frac{\sigma_i \sigma_j }{|\sigma|^2},\ \ \ \ \ \sigma = \nabla u\not= 0,
\]
has eigenvalues contained in the interval
$[\min\{1,p-1\},\max\{1,p-1\}]$.

The equation \eqref{0} originates from the doubly degenerate divergence-form equation
\begin{equation}\label{me}
|\nabla u|^{p-2} u_t = \operatorname{div}(|\nabla u|^{p-2}\nabla u).
\end{equation}

At points where $\nabla u \not= 0$, canceling the common factor $|\nabla u|^{p-2}$ in \eqref{me} gives \eqref{0}; thus the two equations are equivalent for $u \in C^2$ with $\nabla u\not=0$. Our results concern \eqref{0}, rather than \eqref{me}, because the equations differ on the critical set. When $p>2$, both sides of \eqref{me} vanish on $\{\nabla u = 0\}$. In particular, every spatially constant function $u(x,t) = e^{-Kt}$, $K>0$, solves \eqref{me} pointwise, although it violates \eqref{LY} for large $t$; see Remark \ref{R:formulation}. The viscosity interpretation of \eqref{0} excludes these functions: a spatially constant solution must satisfy $u_t = 0$. It also provides the framework for the well-posedness results in \cite{Do, BG, JS}.

In this paper we prove the following nonlinear version of the celebrated result of Li and Yau \cite{LY}.

\begin{thrm}\label{T:LY}
Let $\bM$ be a closed manifold with $\Ric \ge 0$, and let $u>0$ be a viscosity solution of \eqref{0} in $\bM\times (0,\infty)$, continuous up to $t=0$, with $u(\cdot,0) \in C^\infty(\bM)$, $u(\cdot,0)>0$. Then
\begin{equation}\label{LY}
(p-1)\frac{|\nabla u|^2}{u^2} - \frac{u_t}{u} \le \frac{n+p-2}{2(p-1)t}
\end{equation}
on $\bM\times(0,\infty)$, in the sense that $v = \log u$ is a viscosity supersolution of the Hamilton-Jacobi equation
\[
\phi_t - (p-1)|\nabla \phi|^2 + \frac{n+p-2}{2(p-1)\,t} = 0 .
\]
In particular, \eqref{LY} holds in the pointwise sense at every point at which $\log u$ is differentiable, and everywhere if $u$ is smooth.
\end{thrm}

The constant in \eqref{LY} is best possible: no inequality of the form \eqref{LY} with a smaller constant can hold on all closed manifolds with $\Ric \ge 0$, see Proposition \ref{P:sharp}. The extremal profile is the self-similar solution on flat $\Rn$ found in \cite{BG}:
\begin{equation}\label{model}
G_p(x,t) = t^{- \frac{n+p-2}{2(p -1)}} \exp \left({-\frac{|x|^2}{4(p-1)t}}\right).
\end{equation}
This function satisfies \eqref{LY} with equality at every $(x,t) \in \Rn\times(0,\infty)$ and plays the role of the backward Gaussian in the classical Li-Yau theory. Sharpness on closed manifolds then follows by a large-torus limit: the flat tori $\Rn/(L\mathbb Z)^n$ are closed with $\Ric \equiv 0$, and, as the period $L\to\infty$, suitable periodic solutions converge locally on $\Rn$ to the extremal profile, so that no better constant can hold uniformly on the class of closed manifolds with $\Ric \ge 0$ (Proposition \ref{P:sharp}).

 As a consequence of Theorem \ref{T:LY}, we establish the following global optimal Harnack inequality.

\begin{thrm}\label{T:global-harnack}
Under the assumptions of Theorem \ref{T:LY}, for any $0<t_1<t_2$ and $x_1,x_2\in \bM$, we have
\begin{equation}\label{Harnack}
u(x_1,t_1)\le u(x_2,t_2)\left(\frac{t_2}{t_1}\right)^{\!\frac{n+p-2}{2(p-1)}}
\exp\!\Big(\frac{d(x_1,x_2)^2}{4(p-1)(t_2-t_1)}\Big),
\end{equation}
where $d(x_1,x_2)$ denotes the Riemannian distance.
\end{thrm}

When $p = 2$, the estimate \eqref{LY} is the sharp Li-Yau inequality with constant $\frac n2$. For $p\not= 2$,
the normalized equation \eqref{0} arises in connection with the games of tug of war with noise introduced by Peres-Sheffield \cite{PS}, whose value functions are governed by the normalized $p$-Laplacian. For existence, uniqueness, and regularity, we refer to \cite{Do, BG, JS}, and for the parabolic interpretation through games, to \cite{MPR}. Kotschwar-Ni \cite{KN} obtained gradient estimates
for positive $p$-harmonic functions and related nonlinear
parabolic equations, using Bochner computations and
regularization arguments. For the connection between $p$-harmonic functions and inverse mean curvature flow, see also Moser \cite{Mo}. For the normalized equation \eqref{0}, our proof likewise uses uniformly parabolic approximations, followed by viscosity stability, to obtain the estimate on the critical set while preserving the sharp constant. Sharp estimates for quasilinear isotropic operators on manifolds, together with the viscosity techniques used here, are developed in \cite{AC, LTW}; see also \cite{LW, CW}.

We outline the proof of Theorem \ref{T:LY}. Following \cite{LY}, we apply the maximum principle to the functional
\[
F = t\big((p-1)|\nabla v|^2 - v_t\big) - \frac{n+p-2}{2(p-1)}, \qquad v = \log u,
\]
using a Bochner-type identity for the linearized operator
\[
\mathcal L(\xi) = \Delta + (p-2)\sa \nabla^2(\cdot)\,\xi,\xi\da,\ \ \ \ \ \ \ \xi = \frac{\nabla v}{|\nabla v|}.
\]

In Section \ref{S:proof} we show that the quadratic terms arising from differentiation of the direction field $\xi$ cancel exactly. Combining the resulting identity with a generalized Newton inequality (Lemma \ref{L:bochner}) gives the coercivity estimate with the sharp constant. This argument proves Theorem \ref{T:LY} only if $\nabla v \not= 0$ at the maximum point of the functional. There is no reason, however, for the maximum to avoid the critical set, where $\mathcal L(\xi)$ and the drift field are undefined.

Section \ref{S:critical} removes this nondegeneracy assumption. We introduce a family of uniformly parabolic flows in which \emph{only the second-order part} of the operator is regularized,
\[
v_t = \Delta v + (p-2)\,\frac{\nabla^2v(\nabla v,\nabla v)}{|\nabla v|^2+\ve^2} + h_\ve(|\nabla v|^2),
\]
while the first-order term $h_\ve$ will be chosen below. For every choice of $h_\ve$, the Bochner identity of Section \ref{S:proof} remains valid with the single remainder term
\[
h_\ve''(s)\,a^{ij}_\ve s_i s_j,\ \ \ \ \  s = |\nabla v|^2,
\]
see Lemma \ref{L:bochner-eps}. For $p \ge 2$ the choice $h_\ve(s) = (p-1)s$ makes this remainder vanish identically, while for $1<p<2$ a lower-order correction of size $O(\ve^{(p-1)/2})$ compensates exactly for the loss in the Newton inequality near the critical set. Consequently, the approximating flows satisfy the Li-Yau inequality with the sharp constant, uniformly in $\ve$ and without any assumption on the critical set (Theorem \ref{T:LY-eps}). Letting $\ve \to 0$ then gives Theorem \ref{T:LY}.

The regularization $|\nabla u|^{p-2} \rightsquigarrow (|\nabla u|^2+\ve^2)^{\frac{p-2}2}$ goes back to \cite{ES} in the study of motion by mean curvature, to which \eqref{0} formally converges as $p\to 1$. It was used for \eqref{me} in \cite{BG} to construct viscosity solutions and establish comparison principles. This regularization of the weight does not, however, yield the sharp constant in \eqref{LY}; see Remark \ref{R:naive}. In the flows of Section \ref{S:critical}, we choose the first-order term separately from the second-order regularization to retain that constant.
The relation between this choice of regularization
and the computations in \cite{Ga} is discussed in
Remark \ref{R:1993}.

\medskip

To allow negative Ricci curvature and to localize the estimate on noncompact manifolds, we introduce a parameter $\alpha$, following the classical choice $F_\alpha = t(|\nabla v|^2 - \alpha v_t) - \frac{n\alpha^2}2$, $\alpha > 1$, in \cite{LY}. The Bochner identity retains its form (Lemma \ref{L:alpha-id}). On a closed manifold with $\Ric \ge -\kappa$, $\kappa \ge 0$, this gives the Li-Yau inequality
\[
(p-1)\frac{|\nabla u|^2}{u^2} - \alpha\, \frac{u_t}{u} \;\le\; \frac{\alpha^2(n+p-2)}{2(p-1)t}, \qquad \alpha = 1 + c_p \kappa t,
\]
with an explicit constant $c_p$ (see \eqref{cp-def} and Theorem \ref{T:LY-kappa}). As in the classical Li-Yau theory, all of our compact results require only a lower bound for the Ricci curvature.

On a complete noncompact manifold with $\Ric \ge 0$, we prove
the sharp inequality for the approximating flows \eqref{Eeps}
under the geometric hypothesis \Hcut and the polynomial growth
condition \Hgrow. The geometric hypothesis holds in $\Rn$ and,
for $p \ge 2$, on manifolds with nonnegative sectional curvature.
No assumption is imposed on the critical set. The inequality
also passes to locally uniform limits of the approximating
flows; see Theorem \ref{T:noncompact} and
Corollary \ref{C:noncompact-limit}.

The localization argument uses a cutoff of the form
$\phi = \psi^q$, with the integer $q$ chosen larger than the
growth exponent in \Hgrow. Where the cutoff is small, its
decay offsets the allowed growth of the solution's Li-Yau
quantity; elsewhere, the coercive term absorbs the cutoff
errors. This allows us to retain the sharp constant.
The details are given in Remark \ref{R:psiq}.

To obtain the estimate for a prescribed viscosity solution
of \eqref{0}, one must construct approximations satisfying
\Hgrow and identify their limit with that solution.
The comparison principle needed for this identification
is discussed in Remark \ref{R:obstruction}.
Removing the geometric hypothesis \Hcut remains a separate 
question.

\medskip

We briefly discuss the organization of the paper. Section \ref{S:prelim} contains the algebraic preliminaries, including the generalized Newton inequality. In Section \ref{S:proof} we establish the Bochner identity and the coercivity estimate where $\nabla v\not=0$. Section \ref{S:critical} removes this restriction and completes the proof of Theorem \ref{T:LY}. In Section  \ref{S:harnack} we prove Theorem \ref{T:global-harnack}. Section \ref{S:alpha} treats negative Ricci curvature and the noncompact case, proving Theorems \ref{T:LY-kappa} and \ref{T:noncompact}. Finally, we discuss the equality case and open problems in Section \ref{S:final}.

\section{Preliminary material}\label{S:prelim}

We collect here the algebraic facts needed in the proof.
For $\xi\in \mathbb S^{n-1}$, consider the matrix $\mathcal A(\xi) = [a_{ij}(\xi)]\in \operatorname{Sym}(n;\R)$, with entries
\begin{equation}\label{A}
a_{ij}(\xi) = \delta_{ij} + (p-2) \xi_i \xi_j = I + (p-2) \xi \otimes \xi.
\end{equation}
We will denote
\[
\mathcal L(\xi) u = \sum_{i,j=1}^n a_{ij}(\xi) \nabla_{ij} u = \operatorname{tr}(\mathcal A(\xi) \nabla^2 u),
\]
where $\nabla^2 u =[\nabla_{ij} u]$ denotes the Hessian of $u$.
Since $\mathcal A(\xi)\xi = (p-1)\xi$ and $\mathcal A(\xi)v = v$ for $v\perp \xi$, we have $\operatorname{spec}(\mathcal A(\xi)) = \{1,...,1,p-1\}$. Thus $\mathcal A(\xi)$ is uniformly elliptic on the unit sphere: for every $x\in \Rn\setminus\{0\}$ and $\xi\in \mathbb S^{n-1}$,
\begin{equation}\label{ell}
\min\{1,p-1\} |x|^2 \le \sum_{i,j=1}^n a_{ij}(\xi) x_i x_j \le \max\{1,p-1\} |x|^2.
\end{equation}

At points where $\nabla u\not= 0$, the matrix \eqref{A} allows us to write \eqref{0}, with $\xi = \frac{\nabla u}{|\nabla u|}$, as
\[
u_t = \mathcal L(\xi)\, u = \Delta u + (p-2) \nabla_{ij} u \nabla_i u \nabla_j u |\nabla u|^{-2}.
\]

We next prove a generalization of Newton's inequality. For $S = \{v_{ij}\}\in \operatorname{Sym}(n;\R)$ and $\xi\in \mathbb S^{n-1}$, write $|S\xi|^2 = \sa S\xi,S\xi\da$ and $S(\xi,\xi) = \sa S\xi,\xi\da$. We denote by $|S|^2 = \sum_{i,j=1}^n v_{ij}^2$ the squared Hilbert-Schmidt norm of $S$.

\medskip

\noindent \textbf{Notation:} Throughout the paper we use the curvature convention
\[
(\nabla_i\nabla_j-\nabla_j\nabla_i)X_k
=
R_{k\ell ij}X_\ell,
\]
so that
\[
g^{ij}R_{k\ell ij}
=
\operatorname{Ric}_{k\ell}.
\]

\medskip

\begin{lemma}\label{L:bochner}
For any $p>1$ and any $S \in \operatorname{Sym}(n;\R)$ and $\xi\in \mathbb S^{n-1}$, we have
\begin{equation}\label{boc}
|S|^2 + (p-2) |S\xi|^2 \ge \frac{1}{n+p-2} \left[\operatorname{tr}(S) + (p-2) S(\xi,\xi)\right]^2.
\end{equation}
\end{lemma}

\begin{proof}
For $p=2$, this is Newton's inequality. Suppose $p\not= 2$. The inequality \eqref{boc} is unchanged if we replace $S$ by $\hat S = T^\star S T$ and $\xi$ by $\hat \xi = T^\star(\xi)$, with $T\in \mathbb O(n)$. We may therefore assume $\xi = e_1$, so that $S(\xi,\xi) = v_{11}$.
Denote by $s = \operatorname{tr}(S)$ and $q = S(\xi,\xi) = v_{11}$, and write $S = \begin{pmatrix} q & b\\ b^\star & M\end{pmatrix}$, where $b = (v_{1j})$, $j=2,...n$, and $M= [v_{ij}]\in \operatorname{Sym}(n-1;\R)$, with $i, j=2,...,n$. Since the Hilbert-Schmidt norm of $S$ is given by
\[
|S|^2 = q^2 + 2 |b|^2 + |M|^2,
\]
and
\[
|S\xi|^2 = \sa Se_1,Se_1\da = q^2 + |b|^2,
\]
we have
\begin{align*}
|S|^2 + (p-2) |S\xi|^2 & = q^2 + 2 |b|^2 + |M|^2 + (p-2)(q^2 + |b|^2)
\\
& = (p-1) q^2 + p |b|^2 +  |M|^2\ge (p-1) q^2 +  |M|^2.
\end{align*}
If we let $\tau = \operatorname{tr}(M)$, then we have $s = q+\tau$, and Newton's inequality gives
\[
\tau^2 \le (n-1) |M|^2.
\]
We thus find
\[
|S|^2 + (p-2) |S\xi|^2 \ge (p-1) q^2 + \frac{\tau^2}{n-1}.
\]
To prove \eqref{boc}, it therefore suffices to show that
\begin{equation}\label{lagmul}
(p-1) q^2 + \frac{\tau^2}{n-1} \ge \frac{1}{n+p-2} \left[q+\tau + (p-2) q\right]^2 = \frac{\left[\tau + (p-1) q\right]^2}{n+p-2}.
\end{equation}
Multiplying both sides of \eqref{lagmul} by $n+p-2 = (n-1) + (p-1)$, the inequality is reduced to
\[
(n-1) q^2 + \frac{1}{n-1}\tau^2 \ge 2 q \tau,
\]
which proves \eqref{boc}.

\end{proof}

\begin{cor}\label{C:boc}
Let $v\in C^2$. Then, at any point where $\nabla v\not=0$, one has with $\xi = \frac{\nabla v}{|\nabla v|}$:
\begin{equation}\label{boc2}
|\nabla^2 v|^2 + (p-2) |\nabla v|^{-2} \sum_{i=1}^n \left(\sum_{j=1}^n \nabla_{ij} v \nabla_j v\right)^2 \ge \frac{1}{n+p-2} (\mathcal L(\xi)v)^2.
\end{equation}
\end{cor}

\begin{proof}
It suffices to apply Lemma \ref{L:bochner} with $S = \nabla^2 v$, and observe that the quantity within square brackets in the right-hand side of \eqref{boc} is precisely $\mathcal L(\xi)v$.

\end{proof}

\section{The Bochner identity and the coercivity of the Li-Yau functional}\label{S:proof}

Fix $p>1$ and let $u(x,t)>0$ be a smooth solution of \eqref{0}. We work throughout this section at points where $\nabla u \not= 0$. There, \eqref{0} agrees with the divergence-form equation \eqref{me}, which we use for the computations. For a change of variables $u=f(v)$, we have
\begin{equation}\label{pilcomp}
\Delta_p u  = f'(v) |f'(v)|^{p-2} \Delta_p v + (p-1) |f'(v)|^{p-2} f''(v) |\nabla v|^p,
\end{equation}
and
\[
|\nabla u|^{p-2} u_t = f'(v) |f'(v)|^{p-2} |\nabla v|^{p-2} v_t,
\]
so $v$ satisfies
\begin{equation}\label{me2}
f'(v)  |\nabla v|^{p-2} v_t = f'(v)  \Delta_p v + (p-1)  f''(v) |\nabla v|^p.
\end{equation}
If $u>0$ and $f(v) = e^v$, we obtain from \eqref{me2}
\begin{equation}\label{me3}
|\nabla v|^{p-2} v_t =   \Delta_p v + (p-1)  |\nabla v|^p.
\end{equation}
At points where $\nabla v\not= 0$, we can divide both sides of \eqref{me3} by $|\nabla v|^{p-2}$, obtaining with $a_{ij}(\xi)$ as in \eqref{A} and $\xi = \frac{\nabla v}{|\nabla v|}$:
\begin{equation}\label{me4}
(p-1)  |\nabla v|^2 - v_t = - |\nabla v|^{2-p} \Delta_p v = - a_{ij}(\xi) \nabla_{ij} v = - \mathcal L(\xi) v.
\end{equation}

\begin{dfn}\label{D:P}
We define
\[
w:=(p-1)|\nabla v|^2-v_t,\qquad
F:= t\,w-a,\ \ \ \ \ \ a = \frac{n+p-2}{2(p-1)},
\]
and, at points where $\nabla v \not= 0$,
\[
\xi=\frac{\nabla v}{|\nabla v|},\qquad
L:= \mathcal L(\xi) = a_{ij}(\xi)\nabla_{ij},\qquad
A:=(p-2)\,\frac{\nabla^2v(\xi)-\nabla^2v(\xi,\xi)\,\xi}{|\nabla v|}
+(p-1)\nabla v .
\]
\end{dfn}

Observe that, in view of \eqref{me4}, we have $w = -Lv$ at every point where $\nabla v\not=0$, and that Theorem \ref{T:LY} is equivalent to the assertion $F \le 0$.

\begin{lemma}\label{L:crucial}
Let $u>0$ be a smooth solution to \eqref{0}, and let $v = \log u$. Then, at every point where $\nabla v\neq 0$, one has the exact identity
\begin{equation}\label{crucial}
LF+2\langle \nabla F,A\rangle-F_t
=
2(p-1)t\Bigl(|\nabla^2v|^2+(p-2)|\nabla^2v(\xi)|^2+\Ric(\nabla v,\nabla v)\Bigr)
-\frac{F}{t}-\frac{a}{t}.
\end{equation}
In particular, if $\Ric\ge 0$, then
\[
LF+2\langle \nabla F,A\rangle-F_t
\ge
2(p-1)t\Bigl(|\nabla^2v|^2+(p-2)|\nabla^2v(\xi)|^2\Bigr)
-\frac{F}{t}-\frac{a}{t}.
\]
\end{lemma}

\begin{proof}
Fix a point $(x_0,t_0)$ with $\nabla v(x_0,t_0)\neq 0$. Since the statement is tensorial, we may work at $(x_0,t_0)$ in a local orthonormal frame that is geodesic at $x_0$. Throughout the proof we write
\[
s:=|\nabla v|^2,\qquad \xi=\frac{\nabla v}{|\nabla v|},\qquad
B:=\frac{\nabla^2v(\xi)-\nabla^2v(\xi,\xi)\,\xi}{|\nabla v|}.
\]

We start from
\[
F=t\,w-a,
\qquad\text{so that}\qquad
\nabla F=t\,\nabla w,\quad
LF=t\,Lw,\quad
F_t=w+t\,w_t.
\]
Therefore
\[
LF+2\langle \nabla F,A\rangle-F_t
=
t\bigl(Lw-w_t+2\langle \nabla w,A\rangle\bigr)-w,
\]
and it remains to compute $Lw-w_t+2\langle \nabla w,A\rangle$. Since
$w=(p-1)s-v_t$, we have
\begin{equation}\label{split-classical}
Lw-w_t
=
(p-1)\bigl(Ls-s_t\bigr)-\bigl(Lv_t-(v_t)_t\bigr),
\end{equation}
and we compute the two pieces separately.

\medskip
\noindent
\textbf{1. The Bochner-type identity for $s :=|\nabla v|^2$.}
At the chosen point,
\[
s_i=2v_kv_{ki},\qquad
s_{ij}=2v_{ki}v_{kj}+2v_kv_{kij}.
\]
Hence
\[
\frac12\,Ls
=
a^{ij}v_{ki}v_{kj}+v_k\,a^{ij}v_{kij},
\]
where we have abbreviated $a^{ij} = a_{ij}(\xi)$. The first term is immediate:
\[
a^{ij}v_{ki}v_{kj}
=
|\nabla^2v|^2+(p-2)|\nabla^2v(\xi)|^2.
\]
For the second term, commute covariant derivatives on the 1-form $\nabla v$:
\[
v_{kij}=v_{ijk}+R_{k\ell ij}v_\ell.
\]
Since $a^{ij}$ is symmetric in $i,j$, the curvature contribution coming from the $(p-2)\xi_i\xi_j$ part vanishes by the skew-symmetry of the Riemann tensor in $i, j$, while the trace part produces the Ricci term:
\[
v_k\,\delta^{ij}R_{k\ell ij}v_\ell=\Ric(\nabla v,\nabla v).
\]
Moreover,
\[
v_k\,a^{ij}v_{ijk}
=
\langle \nabla(Lv),\nabla v\rangle
-
v_k\,a^{ij}{}_{,k}\,v_{ij}.
\]
Now
\[
a^{ij}{}_{,k}=(p-2)(\xi^i{}_{,k}\xi^j+\xi^i\xi^j{}_{,k}),
\qquad
\nabla_k \xi_i = \frac{v_{ki}}{|\nabla v|} - \frac{v_i \, v_{k\ell}v_\ell}{|\nabla v|^3},
\]
and since $v_{ij}=v_{ji}$,
\[
v_k\,a^{ij}{}_{,k}\,v_{ij}
=
2(p-2)\,\Bigl\langle \nabla^2v(\xi)-\nabla^2v(\xi,\xi)\,\xi,\;\nabla^2v(\xi)\Bigr\rangle
=
2(p-2)\bigl(|\nabla^2v(\xi)|^2-|\nabla^2v(\xi,\xi)|^2\bigr).
\]
Therefore, using $w = -Lv$, and thus $\langle \nabla(Lv),\nabla v\rangle=-\langle \nabla w,\nabla v\rangle$, we obtain
\begin{equation}\label{Ls-classical}
\frac12\,Ls
=
|\nabla^2v|^2+(p-2)|\nabla^2v(\xi)|^2+\Ric(\nabla v,\nabla v)
-\langle \nabla w,\nabla v\rangle
-2(p-2)\bigl(|\nabla^2v(\xi)|^2-|\nabla^2v(\xi,\xi)|^2\bigr).
\end{equation}

\medskip
\noindent
\textbf{2. The time commutator.}
We differentiate the equation $v_t=Lv+(p-1)s$ with respect to $t$. Since the metric is time independent, $\partial_t$ commutes with covariant differentiation, and the only time dependence in the coefficients of $L=a^{ij}(\xi)\nabla_{ij}$ comes from $\xi$. We thus obtain
\[
(v_t)_t = (p-2)(\xi^i\xi^j)_t\, v_{ij} + L(v_t) + (p-1) s_t,
\]
that is,
\begin{equation}\label{Lvt-classical}
Lv_t-(v_t)_t
=
-\,(p-2)(\xi^i\xi^j)_t\, v_{ij} - (p-1)\, s_t .
\end{equation}
Using
\[
\xi_t
=
\frac{\nabla v_t}{|\nabla v|}
-
\frac{\langle \nabla v_t,\xi\rangle}{|\nabla v|}\,\xi,
\]
we find
\[
(p-2)(\xi^i\xi^j)_t v_{ij}
=
2(p-2)\left\langle \frac{\nabla v_t}{|\nabla v|},\,\nabla^2v(\xi)-\nabla^2v(\xi,\xi)\,\xi\right\rangle.
\]
Now differentiate
$v_t=Lv+(p-1)|\nabla v|^2$
in space to obtain
\[
\nabla v_t
=
\nabla(Lv)+2(p-1)\nabla^2v(\nabla v)
=
-\nabla w+2(p-1)\nabla^2v(\nabla v).
\]
Since $\nabla^2v(\nabla v)=|\nabla v|\,\nabla^2v(\xi)$, this yields
\begin{align}\label{axi-t}
(p-2)(\xi^i\xi^j)_t v_{ij}
& =
2(p-2)\bigg(
2(p-1)\bigl(|\nabla^2v(\xi)|^2-|\nabla^2v(\xi,\xi)|^2\bigr)
\\
& -\left\langle \frac{\nabla w}{|\nabla v|},\,\nabla^2v(\xi)-\nabla^2v(\xi,\xi)\,\xi\right\rangle
\bigg).
\notag
\end{align}

\medskip
\noindent
\textbf{3. Combining the two pieces.}
Substituting \eqref{Lvt-classical} into \eqref{split-classical}, the two occurrences of $(p-1)s_t$ cancel, and we are left with
\[
Lw-w_t = (p-1)\, Ls + (p-2)(\xi^i\xi^j)_t\, v_{ij}.
\]
Inserting \eqref{Ls-classical} and \eqref{axi-t}, the quadratic terms $\pm\, 4(p-1)(p-2)\bigl(|\nabla^2v(\xi)|^2-|\nabla^2v(\xi,\xi)|^2\bigr)$ cancel exactly, and we obtain
\begin{align*}
Lw-w_t
&=
2(p-1)\Bigl(|\nabla^2v|^2+(p-2)|\nabla^2v(\xi)|^2+\Ric(\nabla v,\nabla v)\Bigr)
-2\langle \nabla w,(p-2)B+(p-1)\nabla v\rangle \\
&=
2(p-1)\Bigl(|\nabla^2v|^2+(p-2)|\nabla^2v(\xi)|^2+\Ric(\nabla v,\nabla v)\Bigr)
-2\langle \nabla w,A\rangle .
\end{align*}
Multiplying by $t$, using $\nabla F=t\,\nabla w$ and $w=(F+a)/t$, we conclude
\[
LF+2\langle \nabla F,A\rangle-F_t
=
2(p-1)t\Bigl(|\nabla^2v|^2+(p-2)|\nabla^2v(\xi)|^2+\Ric(\nabla v,\nabla v)\Bigr)
-\frac{F}{t}-\frac{a}{t},
\]
as claimed.
\end{proof}

The next result gives the coercivity needed in the proof of Theorem \ref{T:LY}.

\begin{cor}\label{C:coercive}
With $F$ and $A$ as in Definition \ref{D:P}, at every point where $\nabla v \not= 0$ one has
\begin{equation}\label{coercive}
L F + 2 \langle \nabla F, A\rangle - F_t \geq \frac{2(p-1)}{n+p-2}\, \frac{F(F+a)}{t}
+ 2(p-1)t\,\Ric(\nabla v,\nabla v).
\end{equation}
In particular, if $\Ric \ge 0$, at every point where $\nabla v\not=0$ and $F\ge 0$,
\[
L F + 2 \langle \nabla F, A\rangle - F_t \geq  \frac{2(p-1)}{n+p-2} \frac{F^2}{t}.
\]
\end{cor}

\begin{proof}
By Corollary \ref{C:boc} applied with $S = \nabla^2v$ and $\xi = \nabla v/|\nabla v|$, and by \eqref{me4}, which gives $\mathcal L(\xi) v = Lv = -w$, we obtain
\[
|\nabla^2v|^2+(p-2)|\nabla^2v(\xi)|^2
\ge
\frac{1}{n+p-2}\,w^2 = \frac{1}{n+p-2}\,\frac{(F+a)^2}{t^2}.
\]
Substituting into \eqref{crucial} and using $\frac{2(p-1)}{n+p-2}\,a=1$, we find
\begin{align*}
& LF+2\langle \nabla F,A\rangle-F_t
\\
& \ge
\frac{2(p-1)}{n+p-2}\frac{(F+a)^2}{t}
-\frac{F+a}{t}
+ 2(p-1)t\,\Ric(\nabla v,\nabla v)
\\
& =
\frac{2(p-1)}{n+p-2}\,\frac{F(F+a)}{t} + 2(p-1)t\,\Ric(\nabla v,\nabla v).
\end{align*}
The second statement follows since $F+a \ge F \ge 0$.

\end{proof}

Corollary \ref{C:coercive} yields Theorem \ref{T:LY} provided the maximum in the argument below occurs outside the critical set. Suppose that, for some $T > 0$ and $\nu>0$, the function
\[
Q_\nu := F - \nu t
\]
attains a positive maximum at $(x_0,t_0) \in \bM\times(0,T]$. Since $F(\cdot,0) = -a<0$, we have $t_0>0$. The case $t_0 = T$ is included, with a one-sided time derivative. At the maximum point,
\[
\nabla F(x_0,t_0) = 0, \qquad F_t(x_0,t_0) \ge \nu, \qquad \nabla^2F(x_0,t_0)\le 0.
\]
If, in addition, $\nabla v(x_0,t_0)\not=0$, then $L F(x_0,t_0) \le 0$ by \eqref{ell}, and Corollary \ref{C:coercive} gives at $(x_0,t_0)$, keeping in mind that $F(x_0,t_0) > \nu t_0 > 0$,
\[
0 \ge LF - F_t + \nu = LF + 2\sa \nabla F, A\da - F_t + \nu \ge \frac{2(p-1)}{n+p-2} \frac{F^2}{t_0} + \nu > 0,
\]
a contradiction. Letting $\nu \to 0$ would give $F \le 0$, which is \eqref{LY}. The maximum point, however, may lie in the critical set $\{\nabla v = 0\}$, where $L$, $A$, and the identity \eqref{crucial} are undefined. We address this difficulty in the next section.

\section{Removal of the nondegeneracy assumption and proof of Theorem \ref{T:LY}}\label{S:critical}

We complete the proof of Theorem \ref{T:LY} by removing the assumption $\nabla v(x_0,t_0)\not=0$ from the maximum principle argument of Section \ref{S:proof}. We first recall the viscosity interpretation of \eqref{0} in Definition \ref{D:visc} and its equivalent formulation \eqref{visc-flat}. Remark \ref{R:formulation} explains the distinction between this interpretation and the divergence-form equation.

In what follows, given $q\in \Rn\setminus\{0\}$, we let $\hat q := q/|q|$, and consider the function $F: (T_x\bM\setminus\{0\})\times \operatorname{Sym}(T_x\bM)$ given by
\[
F(q,X) := \operatorname{tr}\big(\mathcal A(\hat q)\,X\big)
= \operatorname{tr}X + (p-2)\,X(\hat q,\hat q), 
\]
with $\mathcal A$ as in \eqref{A}. We respectively denote by $F^*$, $F_*$ the upper and the lower semicontinuous envelope of $F$, extended to $q = 0$.
We note that in this notation we can write \eqref{0} as 
\[
u_t = F(\nabla u, \nabla^2u).
\]

\begin{dfn}\label{D:visc}
A function $u \in \mathrm{USC}(\bM\times(0,T))$ is a viscosity subsolution of \eqref{0} if, whenever $\varphi \in C^{2,1}$ touches $u$ from above at $(x_0,t_0)$,
\[
\varphi_t(x_0,t_0) \;\le\; F^*\big(\nabla\varphi(x_0,t_0),\, \nabla^2\varphi(x_0,t_0)\big);
\]
supersolutions are defined symmetrically, with $u\in \mathrm{LSC}(\bM\times(0,T))$, the reverse inequality and $F_*$, and a viscosity solution is a continuous function which is both.
\end{dfn}

The following three properties of Definition \ref{D:visc} will be used repeatedly. 
\begin{itemize}
\item[(a)] \emph{(The envelopes at a critical point.)} For every $X\in \operatorname{Sym}(T_x\bM)$,
\begin{equation}\label{envelopes}
F^*(0,X) = \operatorname{tr}X + (p-2)\lambda_{\max}(X), \qquad
F_*(0,X) = \operatorname{tr}X + (p-2)\lambda_{\min}(X)
\end{equation}
when $p \ge 2$, the roles of $\lambda_{\max}$ and $\lambda_{\min}$ being interchanged when $1<p<2$. Indeed, along any sequence $q_k \to 0$, $q_k \ne 0$, the quantity $X(\hat q_k,\hat q_k)$ has as its set of limit points exactly the interval $[\lambda_{\min}(X),\lambda_{\max}(X)]$. Thus at a critical point of the test function the quotient $\nabla^2\varphi(\nabla\varphi,\nabla\varphi)/|\nabla\varphi|^2$ is replaced by that interval.

\item[(b)] \emph{(Admissibility of the operator.)} Since
\[
\big|(p-2)\,X(\hat q,\hat q)\big| \;\le\; |p-2|\,\|X\| \;\longrightarrow\; 0
\qquad \text{as } X \to 0, \text{ uniformly in } q \neq 0,
\]
one has $F^*(0,0) = F_*(0,0) = 0$. Thus \eqref{0} satisfies the structural conditions, namely degenerate ellipticity, continuity off $\{q = 0\}$, and the coincidence just displayed, under which the classical theory of singular parabolic equations of Ishii-Souganidis \cite{IS}, \cite[Chapter 3]{Gi}, applies through the semicontinuous envelopes $F^*, F_*$. For the adaptation of the notion to Riemannian manifolds see \cite[Section 2]{LTW}, and for \eqref{0} in particular \cite{Do, BG, JS}.

\item[(c)] \emph{(Reduction to vanishing-Hessian test functions.)} By \cite[Section 2]{IS} and \cite[Section 2]{JS}, Definition \ref{D:visc} is unchanged if, at a degenerate contact point, one restricts the test functions to those satisfying $\nabla\varphi(x_0,t_0) = 0$ and $\nabla^2\varphi(x_0,t_0) = 0$, for example, $\varphi = f(d(x,x_0)) + \psi(t)$ with $f$ vanishing to fourth order. In this case, \eqref{envelopes} reduces the sub- and supersolution conditions to
\begin{equation}\label{visc-flat}
\varphi_t(x_0,t_0) \le 0, \qquad \text{respectively} \qquad \varphi_t(x_0,t_0) \ge 0 .
\end{equation}
This is the formulation we shall use in the passage to the limit of Section \ref{SS:limit}: it requires no information on the direction $\nabla\varphi/|\nabla\varphi|$ at the contact point.
\end{itemize}

\begin{rmrk}\label{R:formulation}
For $u \in C^2$ with $\nabla u \neq 0$, Definition \ref{D:visc} agrees with the pointwise interpretation of \eqref{0}, and hence of \eqref{me}. On the critical set, however, the equations are no longer equivalent, and \eqref{me} loses essential information. Let $p>2$, $K>0$, and consider on $\Rn\times(0,\infty)$ the function $u(x,t) = e^{-Kt}$. Since $\nabla u \equiv 0$, both sides of \eqref{me} vanish identically, so that $u$ is a smooth, positive solution of \eqref{me} in the literal pointwise sense. On the other hand, with $v = \log u$,
\[
t\left((p-1)|\nabla v|^2 - v_t\right) = Kt \underset{t \to \infty}{\longrightarrow} \ \infty,
\]
so that \eqref{LY} fails for $t$ large. By contrast, $u = e^{-Kt}$ is \emph{not} a viscosity solution of \eqref{0} in the sense of Definition \ref{D:visc}: testing with $\varphi = u$, which has $\nabla\varphi \equiv 0$ and $\nabla^2\varphi \equiv 0$, the condition \eqref{visc-flat} forces $\varphi_t = 0$. Thus no conclusion of Li-Yau type can hold for \eqref{me} read literally, whereas, as we prove in this section, the sharp inequality \eqref{LY} does hold for the viscosity solutions of \eqref{0}.
\end{rmrk}

\subsection{The approximating flows}\label{SS:appflows}

Henceforth, $s := |\nabla v|^2$, and for $\ve>0$ we let
\begin{equation}\label{sigma}
\sigma := s + \ve^2, \qquad \beta := \frac{s}{\sigma} \in [0,1).
\end{equation}
We define the regularized coefficients
\begin{equation}\label{Aeps}
a^{ij}_\ve = a^{ij}_\ve(\nabla v) := g^{ij} + (p-2)\, \frac{\nabla^i v\, \nabla^j v}{\sigma},
\qquad
L_\ve \phi := a^{ij}_\ve \nabla_{ij} \phi .
\end{equation}
Since the eigenvalues of $a^{ij}_\ve$ are $1$ (with multiplicity $n-1$) and
$1+(p-2)\beta$, and $\beta\in[0,1)$, we have, \emph{uniformly in} $\ve$,
\begin{equation}\label{unif-ell}
\min\{1,p-1\}\, |\zeta|^2 \;\le\; a^{ij}_\ve \zeta_i \zeta_j \;\le\; \max\{1,p-1\}\, |\zeta|^2,
\qquad \zeta \in T_x\bM.
\end{equation}
Unlike the coefficients $a_{ij}(\xi)$ in \eqref{A}, the functions
$q \mapsto a^{ij}_\ve(q)$ are smooth on the whole tangent space, including
$q = 0$.

Next, we choose the first-order term. For $p\ge2$, we may take $h_\ve(s)=(p-1)s$.
For $1<p<2$, the regularized Newton inequality of
Lemma \ref{L:newton-ext} gives a weaker bound, and we correct
this linear term to recover the sharp coercivity estimate
of Proposition \ref{P:coercive-eps}.
Let
\begin{equation}\label{deltaK}
\delta := \frac{p-1}{4}, \qquad K := \frac{2(p-1)(2-p)_+}{n+p-2},
\end{equation}
so that $K = 0$ when $p \ge 2$. We define $h_\ve \in C^\infty([0,\infty))$ by
\begin{equation}\label{heps}
h_\ve(s) := (p-1)s + K \int_0^s \left(\frac{\ve^2}{\tau + \ve^2}\right)^{\delta} d\tau
= (p-1)s + \frac{K\,\ve^{2\delta}}{1-\delta}\left[(s+\ve^2)^{1-\delta} - \ve^{2(1-\delta)}\right].
\end{equation}
Thus $h_\ve(s) = (p-1)s$ when $p\ge 2$, whereas for $1<p<2$
\begin{equation}\label{heps-props}
h_\ve'(s) = (p-1) + K \left(\frac{\ve^2}{\sigma}\right)^{\delta},
\qquad
h_\ve''(s) = -\,\delta K\, \frac{\ve^{2\delta}}{\sigma^{\delta+1}},
\end{equation}
so that 
\[
p-1 \le h_\ve' \le p-1+K,\ \ \ \ \ h_\ve'' \le 0,
\]
and
\begin{equation}\label{heps-conv}
0 \;\le\; h_\ve(s) - (p-1)s \;\le\; \frac{K}{1-\delta}\, \ve^{2\delta}\,(s+\ve^2)^{1-\delta}
\;\underset{\ve\to 0}{\longrightarrow}\; 0
\qquad \text{locally uniformly in } s.
\end{equation}

We consider the approximating flow on $\bM \times (0,T]$:
\begin{equation}\label{Eeps}
\boxed{v_t = L_\ve v + h_\ve(|\nabla v|^2).}
\end{equation}

By \eqref{unif-ell}, equation \eqref{Eeps} is quasilinear and uniformly parabolic, uniformly in $\ve$, with coefficients depending smoothly on $\nabla v$. Thus no singularity arises at critical points.
In terms of $u = e^v$, \eqref{Eeps} reads

\begin{equation}\label{Eeps-u}
u_t = \Delta u + (p-2)\,\frac{\nabla^2u(\nabla u,\nabla u) - |\nabla u|^4/u}{|\nabla u|^2 + \ve^2 u^2}
+ (p-2)\frac{|\nabla u|^2}{u} + u \left[h_\ve\left(\tfrac{|\nabla u|^2}{u^2}\right) - (p-1)\tfrac{|\nabla u|^2}{u^2}\right],
\end{equation}
This equation is invariant under multiplication $u \mapsto \lambda u$, $\lambda > 0$, since \eqref{Eeps} depends only on derivatives of $v = \log u$. By \eqref{heps-conv}, equation \eqref{Eeps-u} converges formally to \eqref{0} as $\ve \to 0$ at points where $\nabla u \not=0$.

In analogy with Definition \ref{D:P}, we set
\begin{equation}\label{Feps-def}
w_\ve := h_\ve(|\nabla v|^2) - v_t \;=\; -L_\ve v,
\qquad
F_\ve := t\, w_\ve - a, \qquad a = \frac{n+p-2}{2(p-1)},
\end{equation}
where $w_\ve = -L_\ve v$ follows from \eqref{Eeps}.
Observe that, since $h_\ve(s) \ge (p-1)s$, we always have
\begin{equation}\label{Feps-dominates}
F_\ve \;\ge\; t\big((p-1)|\nabla v|^2 - v_t\big) - a,
\end{equation}
so that an upper bound on $F_\ve$ implies the corresponding Li-Yau bound.

\begin{rmrk}\label{R:1993}
The choice of the flows \eqref{Eeps} follows a principle used in
the notes \cite{Ga} for quasilinear parabolic equations of the form
\[
a_{ij}(\nabla u)\, \nabla_{ij} u = f(u) + \Phi'(|\nabla u|^2)\, u_t,
\qquad
a_{ij}(\nabla u) = 2\Phi''(|\nabla u|^2)\, u_iu_j + \Phi'(|\nabla u|^2)\,\delta_{ij}
\]
This includes \eqref{me} when $\Phi'(s) = s^{\frac{p-2}2}$ and
$f = 0$. The second-order coefficients are normalized by the structure
function $\Lambda(s) = 2s\,\Phi''(s) + \Phi'(s)$, and the maximum principle
is applied to $P = \Psi(|\nabla u|^2) - 2F(u)$. The choice
$\Psi' = \Lambda$ in \cite{Ga} ensures that the quadratic terms of
highest order cancel after the conditions $\nabla P = 0$, $P_t = 0$
are imposed at the maximum point. Here we choose the first-order term
$h_\ve$ in \eqref{heps} to suit the regularized coefficients $a^{ij}_\ve$:
the remainder in Lemma \ref{L:bochner-eps} vanishes when $p \ge 2$ and
can be absorbed without changing the sharp constant when $1<p<2$.
Thus the second-order coefficients and the scalar functions are chosen
separately, following the same principle as in \cite{Ga}.
\end{rmrk}

\subsection{The Bochner identity for the flow \eqref{Eeps}}

The following lemma extends Lemma \ref{L:crucial} to the regularized flow. The identity holds at every point of $\bM\times(0,T]$, including the critical set, since \eqref{sigma} gives $\sigma \ge \ve^2 > 0$. We allow any smooth function $h$ on $[0,\infty)$ and later take $h = h_\ve$.

\begin{lemma}\label{L:bochner-eps}
Let $\ve>0$, let $h\in C^\infty([0,\infty))$, and let $v$ be a smooth
solution of $v_t = L_\ve v + h(|\nabla v|^2)$ on $\bM\times(0,T]$. With $w := -L_\ve v = h(s) - v_t$, define
\begin{equation}\label{QB-def}
\mathcal Q_\ve := a^{ij}_\ve\, \nabla_{ki} v\, \nabla_{kj} v
= |\nabla^2v|^2 + (p-2)\,\frac{|\nabla^2v(\nabla v)|^2}{\sigma},
\qquad
B_\ve := \frac{\nabla^2v(\nabla v)}{\sigma} - \frac{\nabla^2v(\nabla v,\nabla v)}{\sigma^2}\, \nabla v,
\end{equation}
and the (smooth) drift field
\begin{equation}\label{Aeps-drift}
A_\ve := h'(s)\, \nabla v + (p-2)\, B_\ve .
\end{equation}
Then, at every point of $\bM\times(0,T]$,
\begin{equation}\label{bochner-eps}
L_\ve w - w_t
= 2h'(s)\Big[\mathcal Q_\ve + \Ric(\nabla v,\nabla v)\Big]
+ h''(s)\, a^{ij}_\ve\, s_i s_j
- 2\big\langle \nabla w, A_\ve \big\rangle .
\end{equation}
\end{lemma}

\begin{proof}
Fix a point $(x_0,t_0)$ and compute in a local orthonormal frame geodesic at
$x_0$, writing $v_i, v_{ij}, v_{ijk}$ for the components of the covariant
derivatives of $v$, and $S := \nabla^2 v$. All formulas below are evaluated at
$(x_0,t_0)$.

\medskip
\noindent \textbf{Step 1: the Bochner term.} From $s_i = 2 v_k v_{ki}$ we
obtain $s_{ij} = 2 v_{ki}v_{kj} + 2 v_k v_{kij}$, whence
\[
L_\ve s = 2\,\mathcal Q_\ve + 2 v_k\, a^{ij}_\ve v_{kij}.
\]
As in the proof of Lemma \ref{L:crucial}, commuting covariant derivatives gives $v_{kij} = v_{ijk} + R_{k\ell ij} v_\ell$.
When contracted with $v_k a^{ij}_\ve$, the $g^{ij}$-part of $a^{ij}_\ve$
produces $\Ric(\nabla v,\nabla v)$, while the $(p-2)v_iv_j/\sigma$-part gives zero by
the skew-symmetry of the curvature tensor in $i,j$ (exactly as in Lemma
\ref{L:crucial}). Moreover,
\[
v_k\, a^{ij}_\ve v_{ijk} = \big\langle \nabla(L_\ve v), \nabla v \big\rangle
- v_k\, \nabla_k(a^{ij}_\ve)\, v_{ij},
\]
and from
$\nabla_k (a^{ij}_\ve) = (p-2)\big[\tfrac{v_{ik}v_j + v_i v_{jk}}{\sigma} - \tfrac{v_i v_j s_k}{\sigma^2}\big]$,
using $\langle \nabla s, \nabla v\rangle = 2\, S(\nabla v,\nabla v)$ and $v_{ik}v_k = (S\,\nabla v)_i$,
\[
v_k\, \nabla_k(a^{ij}_\ve)\, v_{ij}
= (p-2)\left[\frac{2|S\,\nabla v|^2}{\sigma} - \frac{2\,S(\nabla v,\nabla v)^2}{\sigma^2}\right]
= 2(p-2)\,\mathcal P_\ve,
\]
where we have set
\[
\mathcal P_\ve := \frac{|S\,\nabla v|^2}{\sigma} - \frac{S(\nabla v,\nabla v)^2}{\sigma^2}.
\]
Since $L_\ve v = -w$, we conclude
\begin{equation}\label{Ls}
L_\ve s = 2\,\mathcal Q_\ve - 2\langle \nabla w, \nabla v\rangle - 4(p-2)\,\mathcal P_\ve + 2\,\Ric(\nabla v,\nabla v).
\end{equation}

\medskip
\noindent \textbf{Step 2: the time commutator.} Differentiating the equation
$v_t = L_\ve v + h(s)$ with respect to $t$ (the metric is time independent,
so $\partial_t$ commutes with covariant differentiation) gives
\[
v_{tt} = (a^{ij}_\ve)_t v_{ij} + L_\ve (v_t) + h'(s) s_t,
\]
i.e.,
\begin{equation}\label{Lvt}
L_\ve (v_t) - v_{tt} = - (a^{ij}_\ve)_t v_{ij} - h'(s) s_t .
\end{equation}
From
\[
(a^{ij}_\ve)_t = (p-2)\big[\tfrac{(v_i)_t v_j + v_i (v_j)_t}{\sigma} - \tfrac{v_iv_j s_t}{\sigma^2}\big],\ \ \ \text{and}\ \ \ 
s_t = 2\langle \nabla v, \nabla v_t\rangle\]
we find
\begin{equation}\label{at}
(a^{ij}_\ve)_t v_{ij}
= 2(p-2)\left[\frac{\langle S\nabla v, \nabla v_t\rangle}{\sigma} - \frac{S(\nabla v,\nabla v)\,\langle \nabla v, \nabla v_t\rangle}{\sigma^2}\right]
= 2(p-2) \big\langle B_\ve, \nabla v_t \big\rangle.
\end{equation}

\medskip
\noindent \textbf{Step 3: conclusion.} Since $w = h(s) - v_t$,
\[
L_\ve w - w_t = h'(s)\big(L_\ve s - s_t\big) + h''(s)\, a^{ij}_\ve s_i s_j - \big(L_\ve(v_t) - v_{tt}\big).
\]
Substituting \eqref{Lvt} and \eqref{at}, the two occurrences of $h'(s)s_t$
cancel and we obtain
\begin{equation}\label{assemble}
L_\ve w - w_t = h'(s)\, L_\ve s + h''(s)\, a^{ij}_\ve s_i s_j
+ 2(p-2)\big\langle B_\ve, \nabla v_t\big\rangle .
\end{equation}
Next, differentiating $v_t = h(s) - w$ in space gives
$\nabla v_t = h'(s)\,\nabla s - \nabla w$, and since $\langle B_\ve, \nabla s\rangle = 2\,\mathcal P_\ve$,
\[
2(p-2)\big\langle B_\ve, \nabla v_t\big\rangle
= 4(p-2)\, h'(s)\,\mathcal P_\ve - 2(p-2)\big\langle B_\ve, \nabla w\big\rangle.
\]
Substituting this identity and \eqref{Ls} into \eqref{assemble}
cancels the terms $\pm\, 4(p-2)h'(s)\mathcal P_\ve$ and gives
\[
L_\ve w - w_t = 2h'(s)\big[\mathcal Q_\ve + \Ric(\nabla v,\nabla v)\big] + h''(s)  a^{ij}_\ve s_is_j
- 2\big\langle \nabla w,\ h'(s) \nabla v + (p-2) B_\ve \big\rangle,
\]
which is \eqref{bochner-eps}.

\end{proof}

\begin{rmrk}\label{R:specialize}
When $\ve \to 0$ and $\nabla v \not= 0$, one has
$B_\ve \to \frac{\nabla^2v(\xi) - \nabla^2v(\xi,\xi)\xi}{|\nabla v|}$ with
$\xi = \nabla v/|\nabla v|$, and, if in addition $h(s) = (p-1)s$, the identity
\eqref{bochner-eps} reduces exactly to Lemma \ref{L:crucial}.
The cancellation of the quadratic terms
$|\nabla^2v(\xi)|^2 - |\nabla^2v(\xi,\xi)|^2$ in Lemma \ref{L:crucial}
also holds in Lemma \ref{L:bochner-eps}, for every $\ve>0$ and every
choice of the first-order term $h$. The additional term
$h''(s)\, a^{ij}_\ve s_i s_j$ vanishes when $h$ is linear, as it is
for $p\ge 2$ by \eqref{heps}.
\end{rmrk}

\subsection{The extended Newton inequality and coercivity}

To replace Corollary \ref{C:boc}, we extend Lemma \ref{L:bochner}
to vectors of length at most one. For the regularized direction
$\eta := \nabla v/\sqrt{\sigma}$ we have $|\eta| < 1$, and the exponent
in the Newton inequality becomes
\begin{equation}\label{pbeta}
p_\beta := 2 + (p-2)\beta \in [\min\{2,p\}, \max\{2,p\}], \qquad \beta = |\eta|^2 .
\end{equation}

\begin{lemma}\label{L:newton-ext}
Let $p>1$, $S \in \operatorname{Sym}(n;\R)$, and $\eta \in \Rn$ with
$|\eta| \le 1$. Then, with $\beta = |\eta|^2$ and $p_\beta$ as in
\eqref{pbeta},
\begin{equation}\label{newton-ext}
|S|^2 + (p-2)\,|S\eta|^2 \;\ge\; \frac{1}{n + (p-2)\beta}\,
\Big[\operatorname{tr} S + (p-2)\, S(\eta,\eta)\Big]^2 .
\end{equation}
In particular, when $p \ge 2$, since $n+(p-2)\beta \le n+p-2$,
\begin{equation}\label{newton-ext-p2}
|S|^2 + (p-2)\,|S\eta|^2 \;\ge\; \frac{1}{n+p-2}\,
\Big[\operatorname{tr} S + (p-2)\, S(\eta,\eta)\Big]^2 .
\end{equation}
\end{lemma}

\begin{proof}
If $\eta = 0$ this is Newton's inequality. If $\eta \not=0$, write
$\eta = \sqrt{\beta}\,\xi$ with $|\xi| = 1$. Then
$|S\eta|^2 = \beta |S\xi|^2$ and $S(\eta,\eta) = \beta\, S(\xi,\xi)$, so
that \eqref{newton-ext} is precisely \eqref{boc} in Lemma \ref{L:bochner}
with $p$ replaced by $p_\beta$. Moreover $p_\beta > 1$: indeed
$p_\beta - 1 = 1 + (p-2)\beta$, and since $p>1$ we have
$(p-2)\beta > -\beta$, whence $p_\beta - 1 > 1-\beta \ge 0$. Lemma
\ref{L:bochner} therefore applies and gives \eqref{newton-ext}.
\end{proof}

Notice that in our application $\eta = \nabla v/\sqrt{\sigma}$, so that
\[
\mathcal Q_\ve = |S|^2 + (p-2)|S\eta|^2,
\qquad
\operatorname{tr} S + (p-2) S(\eta,\eta) = a^{ij}_\ve v_{ij} = L_\ve v = -w_\ve .
\]

For $1<p<2$, the denominator in \eqref{newton-ext} exceeds
$n+p-2$ by $(2-p)(1-\beta) = (2-p)\ve^2/\sigma$. The correction in
\eqref{heps} compensates for this loss. Its negative remainder
$h_\ve''\, a^{ij}_\ve s_is_j$ is absorbed by the excess
$h_\ve' - (p-1)$. These estimates give the following version of
Corollary \ref{C:coercive}, with the sharp constant independent of
$\ve$ and no restriction on the critical set.

\begin{prop}\label{P:coercive-eps}
Let $\ve>0$, $1<p<\infty$, and let $v$ be a smooth solution of \eqref{Eeps}
on $\bM\times(0,T]$, with $\Ric \ge 0$. Then, at every point of
$\bM\times(0,T]$,
\begin{equation}\label{coercive-eps}
L_\ve F_\ve - (F_\ve)_t + 2\big\langle \nabla F_\ve, A_\ve \big\rangle
\;\ge\; \frac{2(p-1)}{n+p-2}\, \frac{F_\ve\,(F_\ve+a)}{t},
\end{equation}
where $F_\ve$ is as in \eqref{Feps-def} and $A_\ve$ as in
\eqref{Aeps-drift} with $h = h_\ve$.
\end{prop}

\begin{proof}
We first claim the pointwise inequality
\begin{equation}\label{claim-pointwise}
2h_\ve'(s)\, \mathcal Q_\ve + h_\ve''(s)\, a^{ij}_\ve s_i s_j
\;\ge\; \frac{2(p-1)}{n+p-2}\; w_\ve^2 .
\end{equation}

\medskip

\noindent \underline{Case $p \ge 2$.} Here $h_\ve(s) = (p-1)s$, so
$h_\ve' = p-1$, $h_\ve'' = 0$, and \eqref{claim-pointwise} follows at once
from \eqref{newton-ext-p2}.

\medskip

\noindent \underline{Case $1<p<2$.} First, since the eigenvalues of
$a^{ij}_\ve$ do not exceed $1$ and $|\nabla s|^2 = 4|S\,\nabla v|^2 \le 4 s\,|S|^2$,
while
$\mathcal Q_\ve \ge \big(1-(2-p)\beta\big)|S|^2 \ge (p-1)|S|^2$, we have
\[
a^{ij}_\ve s_i s_j \le 4 s\, |S|^2 \le \frac{4s}{p-1}\, \mathcal Q_\ve .
\]
Hence, using \eqref{heps-props} and $\delta = \frac{p-1}{4}$, we obtain
\[
2h_\ve'  \mathcal Q_\ve + h_\ve'' a^{ij}_\ve s_is_j
\;\ge\; \left[2h_\ve'(s) - \frac{4s |h_\ve''(s)|}{p-1}\right] \mathcal Q_\ve,
\]
and
\[
\frac{4s |h_\ve''|}{p-1} = \frac{4\delta K}{p-1}\left(\frac{\ve^2}{\sigma}\right)^{\!\delta} \frac{s}{\sigma}
\le K\left(\frac{\ve^2}{\sigma}\right)^{\!\delta},
\]
so that
\begin{equation}\label{bracket}
2h_\ve' \mathcal Q_\ve + h_\ve'' a^{ij}_\ve s_is_j
\;\ge\; \left[2(p-1) + K\left(\frac{\ve^2}{\sigma}\right)^{\!\delta}\right] \mathcal Q_\ve .
\end{equation}
On the other hand, by Lemma \ref{L:newton-ext},
$\mathcal Q_\ve \ge \frac{w_\ve^2}{n+(p-2)\beta}$, and by the definition
\eqref{deltaK} of $K$, we have
\begin{align*}
\frac{2(p-1)}{n+p-2}\,\big[n+(p-2)\beta\big]
& = 2(p-1)\left[1 + \frac{(2-p)(1-\beta)}{n+p-2}\right]
\\
& = 2(p-1) + K\,\frac{\ve^2}{\sigma}
\;\le\; 2(p-1) + K \left(\frac{\ve^2}{\sigma}\right)^{\!\delta},
\end{align*}
where in the last step we have used $\ve^2/\sigma \le 1$ and $\delta < 1$.
Combining this with \eqref{bracket} proves \eqref{claim-pointwise} also in
this case.

\smallskip
Now, by Lemma \ref{L:bochner-eps} (with $h = h_\ve$), $\Ric \ge 0$,
$h_\ve' > 0$, and \eqref{claim-pointwise},
\[
L_\ve w_\ve - (w_\ve)_t + 2\langle \nabla w_\ve, A_\ve\rangle
\;\ge\; \frac{2(p-1)}{n+p-2}\, w_\ve^2 .
\]
Since $F_\ve = t w_\ve - a$, we have 
\[
\nabla F_\ve = t\,\nabla w_\ve,\ \ \ 
L_\ve F_\ve = t\, L_\ve w_\ve,\ \ \ (F_\ve)_t = w_\ve + t (w_\ve)_t,
\]
and
$w_\ve = (F_\ve + a)/t$. Therefore, using
$\frac{2(p-1)}{n+p-2}\, a = 1$,
\[
L_\ve F_\ve - (F_\ve)_t + 2\langle \nabla F_\ve, A_\ve\rangle
\;\ge\; \frac{2(p-1)}{n+p-2}\, \frac{(F_\ve+a)^2}{t} - \frac{F_\ve+a}{t}
= \frac{2(p-1)}{n+p-2}\, \frac{F_\ve (F_\ve+a)}{t},
\]
which is \eqref{coercive-eps}.
\end{proof}

\subsection{The sharp Li-Yau inequality for the flow \eqref{Eeps}}
We now prove the main estimate for solutions of the flow \eqref{Eeps}.

\begin{thrm}\label{T:LY-eps}
Let $\bM$ be a closed manifold with $\Ric \ge 0$, let $\ve>0$, and let $v$
be a smooth solution of \eqref{Eeps} on $\bM\times[0,T]$. Then
\begin{equation}\label{LY-eps}
F_\ve \;\le\; 0 \qquad \text{on } \bM\times[0,T].
\end{equation}
In particular, by \eqref{Feps-dominates}, with $u = e^v$,
\begin{equation}\label{LY-eps-u}
(p-1)\frac{|\nabla u|^2}{u^2} - \frac{u_t}{u} \;\le\; \frac{n+p-2}{2(p-1)\,t}
\qquad \text{on } \bM\times(0,T].
\end{equation}
\end{thrm}

\begin{proof}
Fix $\nu>0$ and set $Q_\nu := F_\ve - \nu t$. Since
$F_\ve(\cdot,0) = -a < 0$, if $Q_\nu$ were positive somewhere then, $\bM$
being closed, $Q_\nu$ would attain a positive maximum at some point
$(x_0,t_0) \in \bM\times(0,T]$. At $(x_0,t_0)$ we have
\[
\nabla Q_\nu = 0, \qquad (Q_\nu)_t \ge 0, \qquad \nabla^2 Q_\nu \le 0
\]
(the inequality $(Q_\nu)_t \ge 0$ holding also when $t_0 = T$, as a
one-sided derivative). Hence $\nabla F_\ve(x_0,t_0) = 0$ and
$(F_\ve)_t(x_0,t_0) \ge \nu$; moreover, since $F_\ve$ and $Q_\nu$ differ
by the spatially constant function $\nu t$, we have $L_\ve F_\ve = L_\ve
Q_\nu = \operatorname{tr}\big(A_\ve\, \nabla^2Q_\nu\big) \le 0$ at
$(x_0,t_0)$, the trace of the product of the positive semidefinite matrix
$A_\ve = [a^{ij}_\ve]$ (by \eqref{unif-ell}) with the negative
semidefinite matrix $\nabla^2Q_\nu(x_0,t_0)$ being nonpositive. No assumption on $\nabla v(x_0,t_0)$ is needed: every quantity in Proposition \ref{P:coercive-eps} is smooth, and \eqref{coercive-eps} holds at every point. Thus, at $(x_0,t_0)$, since $\nabla F_\ve = 0$,
\[
0 \;\ge\; L_\ve F_\ve - (F_\ve)_t + \nu
\;\ge\; \frac{2(p-1)}{n+p-2}\, \frac{F_\ve (F_\ve+a)}{t_0} + \nu .
\]
Since 
\[
F_\ve(x_0,t_0) = Q_\nu(x_0,t_0) + \nu t_0 > 0,
\]
the right-hand side
is strictly positive, a contradiction. Therefore $Q_\nu \le 0$ on
$\bM\times[0,T]$ for every $\nu>0$, and letting $\nu \to 0$ gives
\eqref{LY-eps}.

\end{proof}

\begin{rmrk}\label{R:naive}
Consider the standard regularization
\[
\big(|\nabla u|^2+\ve^2\big)^{\frac{p-2}{2}} u_t
= \operatorname{div}\Big(\big(|\nabla u|^2+\ve^2\big)^{\frac{p-2}{2}} \nabla u\Big).
\]
This does not yield the sharp constant. Writing the equation
for $v = \log u$ as $v_t = L_\ve v + h(|\nabla v|^2)$ gives
$h(s) = s + (p-2)\frac{s^2}{s+\ve^2}$, with
\[
h'(s) = (p-1) - (p-2)\,\frac{\ve^4}{\sigma^2},
\qquad
h''(s) = 2(p-2)\,\frac{\ve^4}{\sigma^3}.
\]
When $p>2$, at points where $s \lesssim \ve^2$ one has
$h'(s) \approx 1 < p-1$: near critical points the regularized flow behaves
like the heat equation, whose Li-Yau constant $\frac n2$
exceeds $a = \frac{n+p-2}{2(p-1)}$ for $p > 2$. The maximum principle
with this $h$ yields only
$F \le \frac{(p-2)(n+p-2)}{2(p-1)}$, uniformly in $\ve$ but not
$F \le 0$. The definitions \eqref{Aeps} and \eqref{heps} instead decouple the
second-order regularization from the first-order term: Lemma \ref{L:bochner-eps} holds for every $h$, and only
the combination $2h'\mathcal Q_\ve + h'' a^{ij}_\ve s_is_j$ needs to
dominate $\frac{2(p-1)}{n+p-2} w^2$. For $p\ge2$ the optimal choice
$h(s) = (p-1)s$ makes the remainder vanish identically; for $1<p<2$ the
correction in \eqref{heps}, of size $O(\ve^{(p-1)/2})$, restores the sharp
constant.
\end{rmrk}

\subsection{Passage to the limit and proof of Theorem \ref{T:LY}}\label{SS:limit}

Theorem \ref{T:LY-eps} gives the Li-Yau bound with a constant
independent of $\ve$. To pass to the limit, we use solvability of the
approximating equations and comparison for the limit equation.

\begin{hyp}[solvability of \eqref{Eeps}]\label{H:solv}
For every $\ve>0$ and every
$v_0 \in C^\infty(\bM)$ there exists a unique smooth solution
$v^\ve$ of \eqref{Eeps} on $\bM\times[0,T]$ with $v^\ve(\cdot,0) = v_0$.
\end{hyp}

\noindent For fixed $\ve$ the equation \eqref{Eeps} is uniformly parabolic with
coefficients depending smoothly on $\nabla v$, and its zeroth-order term
$h_\ve(|\nabla v|^2)$ has natural (quadratic) growth in the gradient;
short-time solvability, together with the a priori estimates (maximum
principle for $v$, Bernstein-type gradient bound, and then interior
Schauder theory) needed to continue the solution up to time $T$, is
provided by the classical quasilinear theory, see \cite[Chapter V]{LSU}
and \cite[Chapters XI-XII]{Lieb}. The functions
$\min v_0$ and $\max v_0$ are solutions, so
$\min v_0 \le v^\ve \le \max v_0$ by comparison.

\begin{hyp}[comparison for the limit equation]\label{H:comp}
If $\overline v$ is a
bounded upper semicontinuous viscosity subsolution, and $\underline v$ a
bounded lower semicontinuous viscosity supersolution, of the normalized
limit equation on $\bM\times(0,T)$ in the sense of Remark
\ref{R:formulation}, with
$\overline v(\cdot,0) \le \underline v(\cdot,0)$, then
$\overline v \le \underline v$; in particular there is a unique viscosity
solution with given continuous initial datum.
\end{hyp}

\noindent In $\Rn$ this is due to \cite{Do, BG}; see also \cite{JS}. On a
closed manifold the same proof runs in local coordinates once the
Euclidean parabolic maximum principle for semijets is replaced by its
Riemannian counterpart, for which see \cite{AFS}, \cite[Section 3]{AJM}
and \cite[Section 2.2]{LTW}. A complete Riemannian theory of
Ishii-Souganidis type, including comparison, Perron existence, and uniqueness for
$u_t + F(\nabla u,\nabla^2 u) = 0$ with $F$ continuous off
$\{\nabla u = 0\}$, degenerate elliptic and invariant under parallel
translation, was developed in \cite{AJM}; the comparison theorem there
additionally requires $F$ to be \emph{geometric}, i.e.
\[
F(\lambda q, \lambda X + \mu\, q\otimes q) = \lambda F(q,X)\ \ \ \  \ \text{for}\ \ 
\lambda > 0, \mu \in \mathbb R,
\]
 a condition which \eqref{0} satisfies
only in the level-set limit $p \to 1$.
The maximum principle for semicontinuous functions on
manifolds and the bounds on the Hessian of $d(x,y)^2$ used there also
enter the comparison argument here. The appendix of \cite{AJM} treats
non-singular equations on a general manifold, including the
approximating flows \eqref{Eeps}.

\begin{thrm}\label{T:LY-complete}
Let $\bM$ be a closed manifold with $\Ric \ge 0$, let $u_0 \in
C^\infty(\bM)$, $u_0 > 0$, and let $u$ be the unique positive viscosity
solution of the normalized equation \eqref{0} on $\bM\times(0,T]$ with
$u(\cdot,0) = u_0$. Then we have for $v = \log u$
\begin{equation}\label{LY-final}
v_t  \ge (p-1)|\nabla v|^2 - \frac{n+p-2}{2(p-1) t}, 
\end{equation}
in the viscosity sense on $\bM\times(0,T)$. In particular \eqref{LY-final}
holds in the pointwise sense at every point at which $v$ is differentiable
in $(x,t)$; if $u$ is smooth, \eqref{LY} holds everywhere on
$\bM\times(0,T]$, without any restriction on the critical set
of $u$.
\end{thrm}

\begin{proof}
Let $v^\ve$ be the solution of \eqref{Eeps} with $v^\ve(\cdot,0) = \psi :=
\log u_0$ given by \Hsolv.
We use the half-relaxed limits of Barles-Perthame
(see \cite[Section 6]{CIL}, \cite{Barles}), which do not require
equicontinuity of $\{v^\ve\}$. For generalized mean curvature flows,
Chen-Giga-Goto \cite{CGG} used comparison and stability in this way,
while Evans-Spruck \cite{ES} used uniform regularity estimates for the
regularized flows. Here a direct application of Krylov-Safonov theory
would require a $\ve$-uniform bound on the term
$h_\ve(|\nabla v^\ve|^2)$, which has quadratic growth in the gradient.
We avoid the need for such a gradient bound for \eqref{Eeps} by using
comparison for the limit equation.

\smallskip
\noindent \emph{Uniform bounds and behavior at $t=0$.} Since constants
solve \eqref{Eeps}, comparison for the (uniformly parabolic, smooth)
equation \eqref{Eeps} gives $\min \psi \le v^\ve \le \max \psi$, uniformly
in $\ve$. Moreover, with
\[
C_* := \sup_{\ve\in(0,1]}\ \sup_{\bM}\ \big| L_\ve \psi + h_\ve(|\nabla \psi|^2)\big|
\;\le\; C(n,p,\psi) < \infty
\]
(finite by \eqref{unif-ell} and $h_\ve(s) \le (p-1)s + Ks$), the functions
$\psi \pm C_* t$ are, respectively, a supersolution and a subsolution of
\eqref{Eeps}, whence
\begin{equation}\label{barriers}
\psi(x) - C_*\, t \;\le\; v^\ve(x,t) \;\le\; \psi(x) + C_*\, t
\qquad \text{on } \bM\times[0,T], \ \text{uniformly in } \ve .
\end{equation}

\smallskip
\noindent \emph{Half-relaxed limits.} Define, for $(x,t) \in \bM\times[0,T)$,
\[
\overline v(x,t) := \limsup_{\ve\to 0,\, (y,s)\to(x,t)} v^\ve(y,s),
\qquad
\underline v(x,t) := \liminf_{\ve\to 0,\, (y,s)\to(x,t)} v^\ve(y,s),
\]
which are finite by the uniform bounds, upper (resp. lower)
semicontinuous, and satisfy $\underline v \le \overline v$ and, by
\eqref{barriers}, $\overline v(\cdot,0) = \underline v(\cdot,0) = \psi$.
We claim that $\overline v$ is a viscosity subsolution, and $\underline v$
a viscosity supersolution, of the normalized limit equation
\[
v_t = \Delta v + (p-2)|\nabla v|^{-2}\nabla^2v(\nabla v,\nabla v) + (p-1)|\nabla v|^2,
\]
in the sense of Remark \ref{R:formulation}. We give the argument for
$\overline v$. Let $\varphi$ touch $\overline v$ from above at
$z_0 = (x_0,t_0)$, $t_0>0$; by the standard perturbation argument
\cite[Lemma 6.1]{CIL} there exist $\ve_j \to 0$ and points $z_j \to z_0$
of local maximum of $v^{\ve_j} - \varphi$ with
$v^{\ve_j}(z_j) \to \overline v(z_0)$. Since $v^{\ve_j}$ is a classical
solution of \eqref{Eeps},
\begin{equation}\label{touch}
\varphi_t(z_j) \;\le\; a^{ij}_{\ve_j}(\nabla\varphi(z_j))\, \nabla_{ij}\varphi(z_j) + h_{\ve_j}\big(|\nabla\varphi(z_j)|^2\big).
\end{equation}
If $\nabla\varphi(z_0) \not= 0$, then $\nabla\varphi(z_j) \not= 0$ for
large $j$, and since 
\[
a^{ij}_\ve(q) \to g^{ij} + (p-2)q^iq^j/|q|^2\ \ \ \ \text{and}\ \ \ 
h_\ve(s) \to (p-1)s
\]
locally uniformly on $\{q \not= 0\}$ (by
\eqref{Aeps} and \eqref{heps-conv}), passing to the limit in \eqref{touch}
gives the subsolution condition at $z_0$. If $\nabla\varphi(z_0) = 0$,
by Remark \ref{R:formulation} we may assume $\nabla^2\varphi(z_0) = 0$,
and we must show $\varphi_t(z_0) \le 0$. By \eqref{unif-ell},
\[
\big|a^{ij}_{\ve_j}(\nabla\varphi(z_j))\, \nabla_{ij}\varphi(z_j)\big|
\le (n+|p-2|)\, \big|\nabla^2\varphi(z_j)\big| \longrightarrow 0,
\qquad
h_{\ve_j}\big(|\nabla\varphi(z_j)|^2\big) \longrightarrow 0,
\]
since $\nabla\varphi(z_j) \to 0$, $\nabla^2\varphi(z_j) \to 0$ and
$0 \le h_\ve(s) \le (p-1+K)s$. Thus \eqref{touch} yields
$\varphi_t(z_0) \le 0$, as required.
The vanishing-Hessian formulation of Remark
\ref{R:formulation} thus allows us to pass to the limit without
controlling the direction $\nabla\varphi(z_j)/|\nabla\varphi(z_j)|$.
The supersolution property of $\underline v$ is proved in the same way.

\smallskip
\noindent \emph{Identification of the limit.}
By definition, $\underline v \le \overline v$. Since the
subsolution $\overline v$ and the supersolution $\underline v$ have
the same initial datum $\psi$, \Hcomp gives the reverse inequality.
Thus $\overline v = \underline v =: v$ is continuous and is the unique
viscosity solution with datum $\psi$, so $v = \log u$.
The standard property of half-relaxed limits \cite[Remark 6.4]{CIL}
then gives
\[
v^\ve \longrightarrow v = \log u \qquad \text{locally uniformly on } \bM\times[0,T).
\]
By Theorem \ref{T:LY-eps}, for every $\ve>0$ we have $F_\ve \le 0$, i.e.,
using \eqref{Feps-dominates},
\[
(v^\ve)_t \;\ge\; (p-1)|\nabla v^\ve|^2 - \frac at
\qquad \text{pointwise on } \bM\times(0,T].
\]
Thus each $v^\ve$ is a (classical, hence viscosity) supersolution of the
first-order equation $\phi_t - (p-1)|\nabla \phi|^2 + \frac at = 0$, whose
Hamiltonian is continuous on $\bM\times(0,T)\times T\bM$. Since
$v^\ve \to v$ locally uniformly, the stability of viscosity supersolutions
of continuous equations under locally uniform convergence, see
\cite[Remarks 6.3 and 6.4]{CIL} and \cite{Barles}, implies that $v$ is a viscosity
supersolution of the same equation on $\bM\times(0,T)$, which is
\eqref{LY-final}.

Finally, if $v$ is differentiable at a point $z_0 = (x_0,t_0)$, there
exists a $C^1$ function touching $v$ from below at $z_0$ with the same
differential (see, e.g., \cite{Barles} or \cite[Section 2]{CIL}), and the
supersolution property gives \eqref{LY-final} at $z_0$ in the classical
sense.

\end{proof}

We now conclude the argument.

\begin{proof}[Proof of Theorem \ref{T:LY}]
The statement is the content of Theorem \ref{T:LY-complete}, valid on
$\bM\times(0,T)$ for every $T>0$, and hence on $\bM\times(0,\infty)$.

\end{proof}

We now prove that the constant $a = \frac{n+p-2}{2(p-1)}$
in \eqref{LY} is optimal. A large-torus limit transfers the sharpness
given by the Euclidean profile \eqref{model} to the closed manifolds
of Theorem \ref{T:LY}.

\begin{prop}[Sharpness of the constant]\label{P:sharp}
We have the following statements:
\begin{itemize}
\item[(i)] The function $G_p$ in \eqref{model} is a positive viscosity
solution of \eqref{0} on $\Rn\times(0,\infty)$, and, with $v = \log G_p$, we have at every $(x,t)\in\Rn\times(0,\infty)$,
\begin{equation}\label{eq-model}
(p-1)|\nabla v|^2 - v_t  =  \frac{n+p-2}{2(p-1) t},
\end{equation}
i.e., $G_p$ verifies \eqref{LY} identically.
\item[(ii)] The constant in \eqref{LY} cannot be improved within the class
of closed manifolds: if $a' \ge 0$ is such that
\[
(p-1)|\nabla v|^2 - v_t \le \frac{a'}t
\]
holds for all solutions as in Theorem
\ref{T:LY}, on all closed $\bM$ with $\Ric \ge 0$, then $a' \ge a$.
\end{itemize}
\end{prop}

\begin{proof}
(i) With $v = \log G_p = - \frac{n+p-2}{2(p -1)} \log t - \frac{|x|^2}{4(p-1)t}$, a direct computation gives
\[
\nabla v = -\frac{x}{2(p-1)t},\ \ \ \ \nabla^2 v = -\frac{I}{2(p-1)t},\ \ \ \ v_t = - \frac{n+p-2}{2(p -1) t} + \frac{|x|^2}{4(p-1)t^2},
\]
so that, at points $x \not= 0$, $\Delta v + (p-2)\nabla^2v(\xi,\xi) = -\frac{n+p-2}{2(p-1)t}$ and
$v_t = \Delta v + (p-2)\nabla^2v(\xi,\xi) + (p-1)|\nabla v|^2$, i.e., $G_p$ solves
\eqref{0} classically on $\{x\not=0\}$; since $\nabla^2 v$ is a multiple of
the identity, the viscosity conditions at $x = 0$ are verified as well.
The identity \eqref{eq-model} follows at once from the two displays.

(ii) The argument is a large-torus limit. For $L \ge 1$ let
$\bM_L := \Rn/(L\,\mathbb Z)^n$ be the flat torus, which is closed with
$\Ric \equiv 0$. Fix $\mu \in (0,1)$ and let $R_\mu > 0$ be such that
$G_p(x,1) < \mu$ for $|x| \ge R_\mu$. Let $g_\mu \in C^\infty(\Rn)$ be a
radial function with
\[
\max\{G_p(\cdot,1),\, \mu\} \;\le\; g_\mu \;\le\; \max\{G_p(\cdot,1),\, \mu\} + \mu,
\qquad
g_\mu \equiv \text{const} \in [\mu, 2\mu] \ \text{outside } B_{R_\mu + 1}
\]
(e.g., a smooth regularization of the maximum of $G_p(\cdot,1)$ and $\mu$);
$g_\mu$ depends only on $\mu$, through $R_\mu$ and the fixed profile $G_p(\cdot,1)$, and not on $L$. Once $\mu$ is fixed, $g_\mu$, its derivatives, and $1/g_\mu$ therefore satisfy bounds that are uniform in every admissible $L$. For $L > 4(R_\mu+1)$, $g_\mu$ defines a smooth positive function on
$\bM_L$. Let $u_{L,\mu}$ be the solution of \eqref{0} on
$\bM_L\times(0,\infty)$ with datum $g_\mu$, given by Theorem
\ref{T:LY-complete} via Hypotheses \ref{H:solv} and \ref{H:comp}; by comparison with constants,
$\mu \le u_{L,\mu} \le \max g_\mu$, uniformly in $L$. Regard the $u_{L,\mu}$ as periodic functions on $\Rn$ and apply the half-relaxed limit argument from the proof of Theorem \ref{T:LY-complete}, now with $(\ve, L^{-1}) \to (0,0)$. The barrier constant $C_*$ in \eqref{barriers} is independent of $L$ because $\log g_\mu$ has $L$-independent $C^2$ bounds. The comparison principle on $\Rn$ \cite{Do, BG} for bounded semicontinuous sub- and supersolutions then gives
\[
u_{L,\mu} \longrightarrow u_\mu \qquad \text{locally uniformly on } \Rn\times[0,\infty),
\]
where $u_\mu$ is the unique bounded viscosity solution of \eqref{0} on
$\Rn$ with datum $g_\mu$. Since the operator in \eqref{0} depends only on
the derivatives of $u$, the equation is invariant under $u \mapsto u + c$;
as $G_p(\cdot,1) \le g_\mu \le G_p(\cdot,1) + 2\mu$, comparison therefore
gives
\begin{equation}\label{squeeze}
G_p(x,1+t) \;\le\; u_\mu(x,t) \;\le\; G_p(x,1+t) + 2\mu .
\end{equation}
Now assume $(p-1)|\nabla v|^2 - v_t \le a'/t$ holds for all solutions as
in Theorem \ref{T:LY} on all closed manifolds with $\Ric \ge 0$; in
particular, each $v_{L,\mu} := \log u_{L,\mu}$ is a viscosity
supersolution of the Hamilton-Jacobi equation
\[
\phi_t - (p-1)|\nabla\phi|^2 + \frac{a'}t = 0.
\]
By stability
under locally uniform convergence, so is $v_\mu := \log u_\mu$, and, by
\eqref{squeeze}, letting $\mu \downarrow 0$ (locally uniform convergence
again, since $G_p(\cdot,1+t)$ is locally bounded away from zero), so is
$v_\infty(x,t) := \log G_p(x,1+t)$. Since $v_\infty$ is smooth, the
inequality holds classically, and \eqref{eq-model} (with $t$ replaced by
$1+t$) gives
\[
\frac{n+p-2}{2(p-1)(1+t)} = (p-1)|\nabla v_\infty|^2 - (v_\infty)_t \;\le\; \frac{a'}{t}
\qquad \text{for all } t>0,
\]
and letting $t \to \infty$ yields $a' \ge \frac{n+p-2}{2(p-1)} = a$.

\end{proof}

\begin{rmrk}\label{R:noncompact}

Lemma \ref{L:bochner-eps} and Proposition \ref{P:coercive-eps}
are local and do not require compactness of $\bM$. On a complete
noncompact manifold with $\Ric \ge 0$, Gaussian barriers or cutoffs
allow one to localize the maximum principle argument, as in \cite{LY}.
Young's inequality absorbs the resulting drift terms into
$2h_\ve'\,\mathcal Q_\ve$. We carry out the localization in
Section \ref{S:alpha}, where we also prove the Li-Yau inequality on
closed manifolds with $\Ric \ge -\kappa$, see Theorem \ref{T:LY-kappa}.
For the flows \eqref{Eeps} on complete noncompact manifolds, we obtain
the sharp inequality without any assumption on the critical set, see Theorem \ref{T:noncompact}.
\end{rmrk}

\section{From Li-Yau to the global Harnack inequality}\label{S:harnack}

Theorem \ref{T:LY} states that $v = \log u$ is a viscosity supersolution of the Hamilton-Jacobi equation with Hamiltonian $H(t,q) = -(p-1)|q|^2 + \frac{a}{t}$. Integrating along curves gives the Harnack inequality in its Hopf-Lax form, with action $\frac{1}{4(p-1)}\int |\dot\gamma|^2$, the Legendre transform of the quadratic part of $H$.

\begin{proof}[Proof of Theorem \ref{T:global-harnack}]
We give the argument when $u$ is smooth; the general case follows by
running it along the approximating solutions $u^\ve = e^{v^\ve}$ of
\eqref{Eeps-u}, for which \eqref{LY-eps-u} holds classically, and then
letting $\ve \to 0$ using the locally uniform convergence established in
the proof of Theorem \ref{T:LY-complete}.

Let $\gamma:[t_1,t_2]\to \bM$ be any smooth curve with $\gamma(t_1)=x_1$ and $\gamma(t_2)=x_2$. Consider the function $\tau \mapsto v(\gamma(\tau),\tau)$. By the chain rule,
\[
\frac{d}{d\tau}v(\gamma(\tau),\tau) = v_t(\gamma(\tau),\tau)+\langle \nabla v(\gamma(\tau),\tau),\dot\gamma(\tau)\rangle.
\]
From \eqref{LY} we have $v_t\ge (p-1)|\nabla v|^2 - \dfrac{n+p-2}{2(p-1)\tau}$. Hence
\[
\frac{d}{d\tau}v(\gamma(\tau),\tau)\ge (p-1)|\nabla v|^2 + \langle \nabla v,\dot\gamma\rangle - \frac{n+p-2}{2(p-1)\tau}.
\]
For a fixed tangent vector $\dot\gamma$, the quadratic form in $\zeta:=\nabla v$ satisfies
\[
(p-1)|\zeta|^2 + \langle \zeta,\dot\gamma\rangle \ge -\frac{|\dot\gamma|^2}{4(p-1)},
\]
as one sees by completing the square (or minimizing in $\zeta$). Therefore
\[
\frac{d}{d\tau}v(\gamma(\tau),\tau) \ge -\frac{|\dot\gamma(\tau)|^2}{4(p-1)} - \frac{n+p-2}{2(p-1)\tau}.
\]
Integrating from $\tau=t_1$ to $\tau=t_2$ we obtain
\[
v(x_2,t_2)-v(x_1,t_1) \ge -\frac{1}{4(p-1)}\int_{t_1}^{t_2}|\dot\gamma(\tau)|^2\,d\tau - \frac{n+p-2}{2(p-1)}\ln\frac{t_2}{t_1},
\]
or equivalently,
\[
v(x_1,t_1)\le v(x_2,t_2) + \frac{n+p-2}{2(p-1)}\ln\frac{t_2}{t_1} + \frac{1}{4(p-1)}\int_{t_1}^{t_2}|\dot\gamma|^2 d\tau.
\]

The minimum of $\int_{t_1}^{t_2}|\dot\gamma|^2 d\tau$ is attained by a minimizing geodesic traversed at constant speed. Such a geodesic exists by completeness. If $\sigma$ has unit speed and length $d = d(x_1,x_2)$, its reparametrization $\gamma(\tau)=\sigma\big(\frac{d}{t_2-t_1}(\tau-t_1)\big)$ satisfies $|\dot\gamma|=\dfrac{d}{t_2-t_1}$ and
\[
\int_{t_1}^{t_2}|\dot\gamma|^2\,d\tau = \frac{d^2}{t_2-t_1}.
\]
Substituting this minimum into the inequality for $v(x_1,t_1)$ and exponentiating, we obtain \eqref{Harnack}.

\end{proof}

\section{Negative curvature and the noncompact case}\label{S:alpha}

In the classical Li-Yau theory \cite{LY}, one passes from closed manifolds with $\Ric \ge 0$ to manifolds with Ricci curvature bounded below, and then to complete noncompact manifolds, by relaxing the functional to
\[
F_\alpha = t\big(|\nabla v|^2 - \alpha\, v_t\big) - \frac{n\alpha^2}{2}, \qquad \alpha > 1:
\]
the additional coercivity of order $\alpha - 1$ absorbs both the negative curvature terms and the errors generated by the localization. 

The same idea applies to the flows of Section \ref{S:critical}. The Bochner identity of Lemma \ref{L:bochner-eps} retains its form (Lemma \ref{L:alpha-id}), while the coercivity estimate gains a nonnegative term proportional to $(\alpha-1)|\nabla v|^2$ (Proposition \ref{P:alpha-coercive}). We thus obtain the Li-Yau inequality on closed manifolds with $\Ric \ge -\kappa$ (Theorem \ref{T:LY-kappa}). On complete noncompact manifolds, the argument gives the sharp inequality for the approximating flows under the hypotheses of Theorem \ref{T:noncompact}, including the growth condition \Hgrow, without any assumption on the critical set. The difficulties that remain in obtaining an unconditional result in the noncompact case are discussed in Remark \ref{R:obstruction}.

\subsection{The $\alpha$-invariant Bochner identity and coercivity}

In the following result we use the notation of Lemma \ref{L:bochner-eps}, and denote
\[
s = |\nabla v|^2,\ \ \ \ \ w = h(s) - v_t = -L_\ve v.
\]

\begin{lemma}[$\alpha$-invariance]\label{L:alpha-id}
Let $\ve>0$, $h \in C^\infty([0,\infty))$, and let $v$ be a smooth solution on $\bM\times(0,T]$ of 
\[
v_t = L_\ve v + h(|\nabla v|^2).
\]
For $\alpha \in \R$ define
\begin{equation}\label{walpha-def}
w_\alpha := h(s) - \alpha\, v_t .
\end{equation}
Then, at every point of $\bM\times(0,T]$, with $\mathcal Q_\ve$, $B_\ve$, and $A_\ve$ defined in \eqref{QB-def} and \eqref{Aeps-drift}, we have
\begin{equation}\label{id-alpha}
L_\ve w_\alpha - (w_\alpha)_t
= 2h'(s)\Big[\mathcal Q_\ve + \Ric(\nabla v,\nabla v)\Big]
+ h''(s)\, a^{ij}_\ve\, s_i s_j
- 2\big\langle \nabla w_\alpha, A_\ve \big\rangle .
\end{equation}
Thus \eqref{bochner-eps} has the same form for $w_\alpha$, with the same drift field.
\end{lemma}

\begin{proof}
Since $w_\alpha = \alpha w - (\alpha-1) h(s)$, Lemma \ref{L:bochner-eps} gives
\begin{align}\label{alpha-start}
L_\ve w_\alpha - (w_\alpha)_t
&= \alpha\big[L_\ve w - w_t\big] - (\alpha-1)\big[L_\ve (h(s)) - (h(s))_t\big] \notag \\
&= \alpha\Big\{2h'\big[\mathcal Q_\ve + \Ric(\nabla v,\nabla v)\big] + h''\, a^{ij}_\ve s_is_j - 2\langle \nabla w, A_\ve\rangle\Big\} 
\\
&\quad - (\alpha-1)\Big\{h'\big(L_\ve s - s_t\big) + h''\, a^{ij}_\ve s_is_j\Big\}. \notag
\end{align}
We compute $L_\ve s - s_t$. Differentiating $v_t = h(s) - w$ in space gives $\nabla v_t = h'\,\nabla s - \nabla w$, so that, with $S = \nabla^2v$ and $\langle \nabla s, \nabla v\rangle = 2S(\nabla v,\nabla v)$, we have
\[
s_t = 2\langle \nabla v, \nabla v_t\rangle = 4h'\, S(\nabla v,\nabla v) - 2\langle \nabla w, \nabla v\rangle .
\]
Combining this with \eqref{Ls}, the terms $\mp 2\langle \nabla w,\nabla v\rangle$ cancel and we obtain
\begin{equation}\label{Ls-st}
L_\ve s - s_t = 2\,\mathcal Q_\ve + 2\,\Ric(\nabla v,\nabla v) - 4(p-2)\,\mathcal P_\ve - 4h'\, S(\nabla v,\nabla v),
\end{equation}
with $\mathcal P_\ve$ as in the proof of Lemma \ref{L:bochner-eps}. Next, from $\nabla w_\alpha = \alpha\, \nabla w - (\alpha-1)h'\,\nabla s$ we obtain
\[
-2\alpha\langle \nabla w, A_\ve\rangle = -2\langle \nabla w_\alpha, A_\ve\rangle - 2(\alpha-1)h'\,\langle \nabla s, A_\ve\rangle,
\]
and, since $\langle \nabla s, \nabla v\rangle = 2S(\nabla v,\nabla v)$ and $\langle \nabla s, B_\ve\rangle = 2\mathcal P_\ve$,
\[
\langle \nabla s, A_\ve\rangle = 2h'\,S(\nabla v,\nabla v) + 2(p-2)\,\mathcal P_\ve .
\]
Substituting the last three displays into \eqref{alpha-start}, the terms $\mp 4(\alpha-1)(h')^2 S(\nabla v,\nabla v)$ cancel, the terms $\mp 4(\alpha-1)(p-2)h'\,\mathcal P_\ve$ cancel, the coefficient of $2h'[\mathcal Q_\ve + \Ric]$ becomes $\alpha - (\alpha-1) = 1$, and so does that of $h''\,a^{ij}_\ve s_is_j$. This proves \eqref{id-alpha}.

\end{proof}

\begin{prop}[$\alpha$-coercivity]\label{P:alpha-coercive}
Let $\alpha \ge 1$, let $h = h_\ve$ be as in \eqref{heps}, and define
\[
a_\alpha := \alpha^2 a, \qquad F_\alpha := t\, w_\alpha - a_\alpha .
\]
Then, at every point of $\bM\times(0,T]$ where $w_\alpha \ge 0$,
\begin{align}\label{alpha-coercive}
& L_\ve F_\alpha - (F_\alpha)_t + 2\big\langle \nabla F_\alpha, A_\ve\big\rangle
\ge \frac{F_\alpha(F_\alpha + a_\alpha)}{\alpha^2 a t}
\\
& + \frac{4(p-1)^2}{n+p-2}\,\frac{\alpha-1}{\alpha^2} |\nabla v|^2\,\big(F_\alpha + a_\alpha\big)
+ 2h_\ve'(s) t \Ric(\nabla v,\nabla v).
\notag
\end{align}
\end{prop}

\begin{proof}
By Lemma \ref{L:alpha-id} and the pointwise inequality \eqref{claim-pointwise},
\[
L_\ve w_\alpha - (w_\alpha)_t + 2\langle \nabla w_\alpha, A_\ve\rangle
\;\ge\; \frac{2(p-1)}{n+p-2}\, w^2 + 2h_\ve'(s)\,\Ric(\nabla v,\nabla v).
\]
Since 
\[
w = \big[w_\alpha + (\alpha-1)h_\ve(s)\big]/\alpha\ \ \ \ \text{and}\ \ \ \  (\alpha-1)^2 h_\ve(s)^2 \ge 0,
\]
we have
\[
w^2  \ge \frac{w_\alpha^2 + 2(\alpha-1) h_\ve(s) w_\alpha}{\alpha^2},
\]
and at points where $w_\alpha \ge 0$ we may further use $h_\ve(s) \ge (p-1)s$ to obtain
\[
w^2 \;\ge\; \frac{w_\alpha^2 + 2(\alpha-1)(p-1)\, s\, w_\alpha}{\alpha^2} .
\]
Multiplying by $t$ and subtracting $w_\alpha$, using 
\[
F_\alpha = tw_\alpha - a_\alpha,\ \ \nabla F_\alpha = t\,\nabla w_\alpha, \ \ w_\alpha = (F_\alpha + a_\alpha)/t,
\]
 and 
 \[
 \frac{2(p-1)}{n+p-2} = \frac 1a,
 \]
we find
\[
\frac{t}{\alpha^2 a} w_\alpha^2 - w_\alpha = \frac{w_\alpha}{\alpha^2 a}\big(t w_\alpha - \alpha^2 a\big) = \frac{F_\alpha (F_\alpha + a_\alpha)}{\alpha^2 a t},
\]
which gives the first term on the right of \eqref{alpha-coercive}. To obtain the second term, retain the cross term in the bound on $w^2$ and multiply by $t/a$:
\[
\frac{t}{a} w^2  \ge \frac{t}{a\alpha^2}\,w_\alpha^2 \;+\; \frac{2(\alpha-1)(p-1)}{a\alpha^2}\, t\,s\, w_\alpha ,
\]
so that, after subtracting $w_\alpha$ and using the preceding identity, the cross term
\[
\frac{2(\alpha-1)(p-1)}{a\alpha^2}\, t\,s\, w_\alpha
\]
remains. Substituting $t\,w_\alpha = F_\alpha + a_\alpha$ and $s = |\nabla v|^2$, and using $\dfrac{2(p-1)}{a} = \dfrac{4(p-1)^2}{n+p-2}$, we obtain
\[
\frac{2(\alpha-1)(p-1)}{a\alpha^2}\, t\,s\, w_\alpha
= \frac{2(\alpha-1)(p-1)}{a\alpha^2}\, |\nabla v|^2\,\big(F_\alpha + a_\alpha\big)
= \frac{4(p-1)^2}{n+p-2}\,\frac{\alpha-1}{\alpha^2}\, |\nabla v|^2\,\big(F_\alpha + a_\alpha\big),
\]
which is the second term in \eqref{alpha-coercive}. Keeping the Ricci term $2th_\ve'(s)\Ric(\nabla v,\nabla v)$ and using
\[
LF_\alpha + 2\langle \nabla F_\alpha, A_\ve\rangle - (F_\alpha)_t = t\big(L_\ve w_\alpha - (w_\alpha)_t + 2\langle \nabla w_\alpha, A_\ve\rangle\big) - w_\alpha,
\]
 proves \eqref{alpha-coercive}.

\end{proof}

\subsection{Closed manifolds with $\Ric \ge -\kappa$}

\hypersetup{linkcolor=red,citecolor=red,urlcolor=red}
We first extend Theorem \ref{T:LY-complete} to a lower Ricci bound. The argument at the critical set is unchanged. Set
\begin{equation}\label{cp-def}
c_p := 1 \ \ \text{for } p \ge 2, \qquad c_p := 1 + \frac{2(2-p)}{n+p-2} \ \ \text{for } 1<p<2,
\end{equation}
and note that \eqref{deltaK} and \eqref{heps-props} give the exact relation
\begin{equation}\label{hprime-cp}
\sup_{s \ge 0}\, h_\ve'(s) \;=\; p-1+K \;=\; (p-1)\, c_p .
\end{equation}

\begin{thrm}[Li-Yau inequality for $\Ric \ge -\kappa$]\label{T:LY-kappa}
Let $\bM$ be a closed manifold with $\Ric \ge -\kappa\, g$ for some $\kappa \ge 0$, and let $u$ be as in Theorem \ref{T:LY-complete}. Then, for every $t \in (0,T)$, with
\[
\alpha = \alpha(t) := 1 + c_p\,\kappa\, t,
\]
one has, in the viscosity sense (hence pointwise wherever $v = \log u$ is differentiable),
\begin{equation}\label{LY-kappa}
(p-1)|\nabla v|^2 - \alpha\, v_t \;\le\; \frac{\alpha^2\, (n+p-2)}{2(p-1)\, t}.
\end{equation}
When $\kappa = 0$ this is \eqref{LY-final}.
\end{thrm}

\begin{proof}
It suffices to prove that, for every $\ve > 0$, every $T_0 \in (0,T)$, and $\alpha := 1 + c_p\kappa T_0$, the solutions $v = v^\ve$ of \eqref{Eeps} satisfy $F_\alpha \le 0$ on $\bM\times[0,T_0]$; one then applies the estimate at time $t = T_0$, and passes to the limit $\ve \to 0$ exactly as in the proof of Theorem \ref{T:LY-complete} (each $v^\ve$ being a classical supersolution of the Hamilton-Jacobi equation $\phi_t - \frac{p-1}{\alpha}|\nabla \phi|^2 + \frac{\alpha a}{t} = 0$).

Fix $\nu > 0$ and let $Q_\nu := F_\alpha - \nu t$. Since $F_\alpha(\cdot,0) = -a_\alpha < 0$, if $Q_\nu$ were positive somewhere, it would attain a positive maximum at some $(x_0,t_0) \in \bM\times(0,T_0]$, where
\[
\nabla F_\alpha = 0, \qquad (F_\alpha)_t \ge \nu, \qquad L_\ve F_\alpha \le 0,
\]
and $F_\alpha(x_0,t_0) > \nu t_0 > 0$, so that $w_\alpha = (F_\alpha + a_\alpha)/t_0 > 0$. By Proposition \ref{P:alpha-coercive} and $\Ric \ge -\kappa g$, at $(x_0,t_0)$,
\[
0 \;\ge\; \frac{F_\alpha(F_\alpha+a_\alpha)}{\alpha^2 a t_0} + s\left[\frac{4(p-1)^2}{n+p-2}\frac{\alpha-1}{\alpha^2}(F_\alpha+a_\alpha) - 2h_\ve'(s)\,\kappa\, t_0\right] + \nu .
\]
Since $F_\alpha + a_\alpha \ge a_\alpha = \alpha^2 a$, the definition of $a$ gives
\[
\frac{4(p-1)^2}{n+p-2}\,\frac{\alpha-1}{\alpha^2}\,(F_\alpha + a_\alpha)
\;\ge\; \frac{4(p-1)^2 a}{n+p-2}\,(\alpha-1)
\;=\; 2(p-1)(\alpha - 1),
\]
while, by \eqref{hprime-cp}, $2h_\ve'(s)\kappa t_0 \le 2(p-1)c_p \kappa t_0 \le 2(p-1)(\alpha-1)$, because $\alpha - 1 = c_p\kappa T_0 \ge c_p \kappa t_0$. Hence the bracket is nonnegative, and we conclude $0 \ge \nu > 0$, a contradiction. Thus $Q_\nu \le 0$ for every $\nu>0$, i.e., $F_\alpha \le 0$ on $\bM\times[0,T_0]$, which by \eqref{Feps-dominates} (whose proof applies verbatim to $w_\alpha$, since $h_\ve \ge (p-1)s$ and $\alpha > 0$) yields \eqref{LY-kappa} for $v^\ve$ at $t = T_0$.
\end{proof}

\begin{rmrk}

The identities $\frac{4(p-1)^2 a}{n+p-2} = 2(p-1)$ and \eqref{hprime-cp} determine the curvature coefficient in the proof. In particular, $c_p$ in \eqref{cp-def} is the largest ratio $h_\ve'/(p-1)$ arising from the correction \eqref{heps} when $1<p<2$. For $p \ge 2$ one may take $\alpha = 1+\kappa t$. We have not sought the optimal dependence on $\kappa$. The known refinements of \eqref{LY-kappa} for $p=2$ suggest a corresponding question for general $p$.

\end{rmrk}

\subsection{The scope of Theorem \ref{T:LY-kappa}}\label{SS:kappa-scope}

The lower Ricci bound in Theorem \ref{T:LY-kappa} places no restriction on a closed manifold. The examples below show that the result applies to metrics not covered by Theorem \ref{T:LY}.

\smallskip
\noindent\textbf{(A) The hypothesis $\Ric \ge -\kappa g$ is automatic on a closed manifold.} Since $\bM$ is closed, the function
\[
x \mapsto \min\{\Ric_x(\zeta,\zeta) : \zeta \in T_x\bM,\ |\zeta| = 1\}
\]
is continuous on a compact space, hence bounded below; consequently
\[
\Ric \;\ge\; -\kappa\, g \qquad \text{with} \qquad
\kappa := -\min_{x \in \bM}\ \min_{|\zeta| = 1} \Ric_x(\zeta,\zeta) \;<\; \infty
\]
holds on every closed Riemannian manifold. Theorem \ref{T:LY-kappa} therefore applies, with $\kappa$ and $\alpha(t) = 1 + c_p\kappa t$ depending on the metric. Theorem \ref{T:LY} is the case $\Ric \ge 0$.

\smallskip
\noindent\textbf{(B) The hypothesis is genuinely broader than $\Ric \ge 0$.} The following examples show both concretely and generally that many manifolds satisfy $\Ric \ge -\kappa g$ but not $\Ric \ge 0$.

\smallskip
\noindent\textbf{(B) Metrics with negative Ricci curvature.} The following examples satisfy $\Ric \ge -\kappa g$ but not $\Ric \ge 0$.

\begin{itemize}
\item[(i)] \emph{Explicit examples.} A closed hyperbolic $n$-manifold, that is, a compact quotient of $\mathbb H^n_{\mathbb R}$ by a torsion-free discrete group of isometries, satisfies $\Ric = -(n-1)g$; here $\kappa = n-1$, and $\Ric \ge 0$ fails at every point. For a closed quotient of complex hyperbolic space $\mathbb H^m_{\mathbb C}$, with the metric normalized so that the holomorphic sectional curvature is $-4$, one has $\Ric = -2(m+1)g$, so that $\kappa = 2(m+1)$ and $n = 2m$. Products such as $\bM_1 \times \bM_2$, with $\bM_1$ closed hyperbolic and $\bM_2$ a flat torus, satisfy $\Ric \ge -\kappa g$ with the same $\kappa$ as $\bM_1$, while having $\Ric$ neither nonnegative nor everywhere negative.
\item[(ii)]
 \emph{No topological obstruction, in general.} Lohkamp's theorem \cite{Lo} gives a complete metric with negative Ricci curvature on every manifold of dimension $n \ge 3$. In particular, every closed manifold of dimension $n \ge 3$ carries a metric with $\Ric < 0$ everywhere. Such a metric is covered by Theorem \ref{T:LY-kappa} with some $\kappa>0$, but not by Theorem \ref{T:LY}. In dimension two, the Gauss-Bonnet theorem implies that a closed surface admits a metric with $\Ric < 0$ only when its genus is at least two.
\end{itemize}

\subsection{Complete noncompact manifolds: the localized estimate}

We now apply the estimate for $F_\alpha$ with a cutoff. The operator $L_\ve$ is not in divergence form, so the localization requires an upper bound involving the full Hessian of the distance function $r(x) = d(x,x_0)$. The bound on $\Delta r$ supplied by $\Ric \ge 0$ is not sufficient. We therefore assume the following.

\begin{hyp}[distance comparison]\label{H:cut}
There exists $C_0 \ge 0$ such that, in the barrier sense on $\bM \setminus B(x_0,1)$,
\[
\Delta r + (p-2)\, \nabla^2 r(\eta,\eta) \;\le\; \frac{C_0}{r} \qquad \text{for every unit vector } \eta .
\]
\end{hyp}

\noindent \Hcut supplies precisely the comparison bound needed when $\mathcal L(\xi)$ and $L_\ve$ are applied to the distance function.
\begin{itemize}
\item[(i)] It always holds on $\bM = \Rn$, where $\nabla^2 r = \frac 1r(g - dr\otimes dr)$. \item[(ii)] When $p = 2$ it reduces to $\Delta r \le C_0/r$, which is the Laplacian comparison theorem under $\Ric \ge 0$: for $p=2$ no assumption beyond $\Ric \ge 0$ is made, consistently with the classical theory \cite{LY}. \item[(iii)] When $p > 2$ it holds as soon as $\operatorname{sec}_{\bM} \ge 0$, by the Hessian comparison theorem $\nabla^2 r \le \frac 1r\, g$ (combined with Calabi's trick at the cut locus, which is available here because only \emph{upper} barriers for $r$ are required).
\item[(iv)] When $1<p<2$, \Hcut instead requires a \emph{lower} Hessian bound $\nabla^2 r \ge -\frac{C}{r}\, g$, which is a more stringent requirement (it fails, in the barrier sense, at the cut locus); in this case \Hcut should be regarded as a working hypothesis, satisfied on $\Rn$ and, more generally, on manifolds with a pole and appropriately pinched curvature.
\end{itemize}

The compact results, Theorems \ref{T:LY}, \ref{T:LY-complete}, and \ref{T:LY-kappa}, require only a Ricci lower bound; \Hcut enters through localization. For $p$-harmonic functions, Kotschwar-Ni \cite{KN} obtained a local gradient estimate under a lower sectional-curvature bound. Wang-Zhang \cite{WZ} subsequently obtained an estimate depending only on the Ricci lower bound, using Moser iteration and the divergence structure of the $p$-Laplacian. This suggests that integral methods may also remove \Hcut here. They would need to be applied to the regularized flows \eqref{Eeps}, since the weak formulation of \eqref{me} does not provide the required information on the critical set; see Remark \ref{R:formulation}.

All Bochner computations in this section apply pointwise to smooth solutions of the uniformly parabolic flows \eqref{Eeps}.
The maximum principle also requires a growth condition on the Li-Yau quantity. Fix $x_0 \in \bM$ and write $r(x) = d(x,x_0)$.

\begin{hyp}[growth of the Li-Yau quantity]\label{H:grow}
There exist $C_F \ge 0$ and $N \ge 0$ such that, on $\bM\times(0,T]$,
\[
t\,\big(h_\ve(|\nabla v|^2) - v_t\big) \;+\; t\,(v_t)^- \;\le\; C_F\,\big(1 + r(x)\big)^{N}.
\]
\end{hyp}

\noindent \Hgrow bounds the growth of $t(h_\ve(s)-v_t)$ and $t\,(v_t)^-$, while preserving the cancellation between $|\nabla v|^2$ and $v_t$. It imposes no separate upper or lower bound on $|\nabla v|$ and no bounded geometry assumption; the critical set remains unrestricted. The conclusion of Theorem \ref{T:noncompact} implies \Hgrow with $N=0$, since $F_\ve \le 0$ gives both $t(h_\ve(s)-v_t) \le a$ and $t\,(v_t)^- \le a$. The extremal profile \eqref{model} satisfies the unregularized condition with $N=0$, because $t\big((p-1)|\nabla v|^2-v_t\big)\equiv a$ and $t(v_t)^-\le a$. The constants $C_F$ and $N$ may depend on $\ve$; the conclusion $F_\ve\le0$ holds for each fixed $\ve$.

\begin{thrm}[sharp Li-Yau on complete noncompact manifolds]\label{T:noncompact}
Let $\bM$ be complete noncompact with $\Ric \ge 0$, satisfying \Hcut. Let $\ve \in (0,1]$ and let $v$ be a smooth solution of \eqref{Eeps} on $\bM\times(0,T]$ satisfying \Hgrow. Then
\[
F_\ve \;\le\; 0 \qquad \text{on } \bM\times(0,T];
\]
in particular, by \eqref{Feps-dominates},
\[
(p-1)|\nabla v|^2 - v_t \;\le\; \frac{n+p-2}{2(p-1)\,t} \qquad \text{on } \bM\times(0,T].
\]
No assumption is made on the critical set of $v$.
\end{thrm}

\begin{proof}

We combine the $\alpha$-coercivity of Proposition \ref{P:alpha-coercive} with a cutoff argument on a large ball $B_R$. The proof has four steps. First, we retain in the coercivity inequality a small multiple of $|\nabla^2v|^2$ to absorb the cutoff errors. Second, we construct $\phi$ and estimate $L_\ve\phi$ using \Hcut. Third, we apply the maximum principle to $G := \phi F_{\alpha,\theta}$, separating the cases in which $\phi$ is large or small at the maximum point and using \Hgrow in the latter case. Finally, we let $R\to\infty$. Throughout, fix $\alpha \in (1,2]$, $\theta \in \big(0,\tfrac{p-1}2\big]$, an integer $q \ge \max\{2, N+1\}$, and $R \ge 1$.

We apply Proposition \ref{P:alpha-coercive} with a cutoff on $B_R$, retaining a small multiple of $|\nabla^2v|^2$ to absorb the cutoff errors. Hypothesis \Hcut bounds $L_\ve\phi$, while \Hgrow controls $G := \phi F_{\alpha,\theta}$ where $\phi$ is small. We then let $R\to\infty$. Fix $\alpha \in (1,2]$, $\theta \in \big(0,\tfrac{p-1}2\big]$, an integer $q \ge \max\{2, N+1\}$, and $R \ge 1$.

\medskip

\noindent \underline{Step 1:} \emph{coercivity with a retained Hessian term.} Splitting 
\[
2h_\ve'\,\mathcal Q_\ve = 2(h_\ve'-\theta)\mathcal Q_\ve + 2\theta\,\mathcal Q_\ve
\]
in the proof of Proposition \ref{P:alpha-coercive}, and using $\mathcal Q_\ve \ge m|S|^2$ with $m := \min\{1,p-1\}$, $S = \nabla^2 v$, the proof of \eqref{claim-pointwise} goes through with $p-1$ replaced by $p-1-\theta$ in the coercive constant: for $1<p<2$ the required inequality becomes 
\[
K(\ve^2/\sigma)^\delta \ge \frac{p-1-\theta}{p-1}\, K\,\ve^2/\sigma,
\]
which holds since $\ve^2/\sigma \le 1$ and $\delta<1$. Hence, at every point of $\bM\times(0,T]$ where $w_\alpha \ge 0$,
\begin{equation}\label{ncp-coercive}
L_\ve F_{\alpha,\theta} - (F_{\alpha,\theta})_t + 2\langle \nabla F_{\alpha,\theta}, A_\ve\rangle
\;\ge\; \frac{F_{\alpha,\theta}(F_{\alpha,\theta}+a_{\alpha,\theta})}{\alpha^2 a_\theta\, t}
+ c_{\alpha,\theta}\; s\, \big(F_{\alpha,\theta} + a_{\alpha,\theta}\big)
+ 2\theta\, m\, t\, |S|^2,
\end{equation}
where
\begin{align*}
a_\theta &:= \frac{n+p-2}{2(p-1-\theta)},
\quad
a_{\alpha,\theta} := \alpha^2 a_\theta,
\quad
F_{\alpha,\theta} := t\, w_\alpha - a_{\alpha,\theta},
\\
c_{\alpha,\theta} &:= \frac{4(p-1)(p-1-\theta)}{n+p-2}\,\frac{\alpha-1}{\alpha^2}
\ge \frac{(p-1)^2(\alpha-1)}{2(n+p-2)}.
\end{align*}
Note that, by \Hgrow and $1 < \alpha \le 2$,
\begin{equation}\label{G-consequence}
F_{\alpha,\theta} \;\le\; t\big(h_\ve(s) - v_t\big) + (\alpha-1)\, t\,(v_t)^- \;\le\; C_F\,(1+r)^N
\qquad \text{on } \bM\times(0,T].
\end{equation}

\medskip

\noindent \underline{Step 2: }\emph{the cutoff.} Let $\psi:[0,\infty)\to[0,1]$ be smooth, $\psi \equiv 1$ on $[0,1]$, $\psi \equiv 0$ on $[2,\infty)$, $|\psi'| + |\psi''| \le C$, and set
\[
\phi := \big[\psi(r/R)\big]^{\,q},
\]
so that $\phi \equiv 1$ on $\overline B_R$ and decays only in the annulus $R \le r \le 2R$. Since $\nabla\phi = q \psi^{q-1}\psi' \frac{\nabla r}{R}$ and $\psi^2 = \phi^{2/q}$, we have
\begin{equation}\label{bounds}
|\nabla \phi|^2 = \frac{q^2\,|\psi'|^2}{R^2}\,\phi^{\,2-\frac 2q} \le \frac{C q^2}{R^2}\,\phi, \qquad \frac{|\nabla\phi|^2}{\phi^2} = \frac{q^2\,|\psi'|^2}{R^2\,\phi^{2/q}} \le \frac{Cq^2}{R^2\,\phi^{2/q}}.
\end{equation}
The bound $Cq^2R^{-2}\phi^{-2/q}$ for $|\nabla\phi|^2/\phi^2$ will control the term multiplying $G^2$.
We now bound $L_\ve \phi$ from below; this is the only property of $\phi$ that will be used in Step 3. Since $\phi = f(r)$ with $f(\rho) := \psi(\rho/R)^q$, the chain rule gives, wherever $r$ is smooth,
\[
L_\ve \phi = f''(r)\, a^{ij}_\ve r_i r_j + f'(r)\, L_\ve r,
\qquad
L_\ve r := \Delta r + (p-2)\, \nabla^2 r(\eta,\eta),
\]
with $\eta := \nabla v/\sqrt\sigma$ the same regularized direction field entering $L_\ve$. By the uniform ellipticity \eqref{unif-ell}, applied to the unit vector $\nabla r$, the first coefficient satisfies $a^{ij}_\ve r_i r_j \in [\min\{1,p-1\},\max\{1,p-1\}]$, while $|f''(r)| \le Cq^2/R^2$ by the same bounds on $\psi',\psi''$ used in \eqref{bounds}. Since $\psi' \le 0$, we have $f'(r) \le 0$; consequently, multiplying an \emph{upper} bound on $L_\ve r$ by $f'(r)$ produces a \emph{lower} bound on the term $f'(r)\,L_\ve r$, and hence on $L_\ve\phi$. We are thus left with bounding $L_\ve r$ from above.

To combine Laplacian comparison under $\Ric \ge 0$ with \Hcut, note that \eqref{sigma} gives $|\eta|^2 = \beta := s/\sigma \in [0,1)$. Since \Hcut is stated for unit vectors, we write $L_\ve r$ as the convex combination
\[
\Delta r + (p-2)\nabla^2 r(\eta,\eta)
= (1-\beta)\,\Delta r + \beta\Big[\Delta r + (p-2)\nabla^2r\big(\tfrac\eta{|\eta|},\tfrac\eta{|\eta|}\big)\Big],
\]
and the two comparison bounds give
\[
L_\ve r \;\le\; (1-\beta)\,\frac{n-1}{r} + \beta\, \frac{C_0}{r} \;\le\; \frac{n-1+C_0}{r}.
\]
At the cut locus, the bound is understood in the barrier sense. Calabi's device replaces $r$ by a smooth upper barrier $r_\delta$ agreeing with $r$ at $x_1$, to which both comparison bounds apply. The estimates for $f'$ and $f''$ then give
\begin{equation}\label{Lphi-ncp}
L_\ve \phi \;\ge\; -\,\frac{C_1(n,p)\, q^2}{R^2} .
\end{equation}

\medskip

\noindent \underline{Step 3:} \emph{the two Cases A and B.} Let $G := \phi\, F_{\alpha,\theta}$, and note that $G = F_{\alpha,\theta}$ on $\overline B_R \times [0,T]$, where $\phi \equiv 1$. Let 
\[
\mathfrak M := \max_{\bar B_{2R}\times[0,T]} G.
\]
If $\mathfrak M \le 0$, then $F_{\alpha,\theta} \le 0$ on $B_R\times(0,T]$ and there is nothing to prove in Step 4 below, so assume $\mathfrak M > 0$, attained at $(x_1,t_1)$. Since 
\[
F_{\alpha,\theta}(\cdot,0) = -a_{\alpha,\theta} < 0,
\]
we have $t_1 > 0$, and $x_1$ lies in the interior of $\{\phi > 0\}$. At $(x_1,t_1)$ we have: 
\begin{align*}
& \nabla G = 0,\ \ (\text{hence}\ \nabla F_{\alpha,\theta} = -F_{\alpha,\theta}\frac{\nabla\phi}{\phi}),\ \ \  
G_t \ge 0;\ \ \  L_\ve G \le 0. 
\end{align*}
Also,
\[
 w_\alpha = \frac{F_{\alpha,\theta} + a_{\alpha,\theta}}{t_1} > 0,
 \]
 so \eqref{ncp-coercive} applies. Using
\[
|A_\ve| \le h_\ve'(s)\sqrt s + |p-2|\,|B_\ve|,
\qquad
|B_\ve| \le \frac{2|S|\sqrt s}{\sigma} \le \frac{2|S|}{\sqrt\sigma} \le \frac{2|S|}{\ve},
\]
and multiplying through by $\phi^2$ exactly as in the compact case, we obtain at $(x_1,t_1)$
\begin{align*}
0 & \ge \frac{G^2}{\alpha^2 a_\theta t_1}
+ c_{\alpha,\theta} s \phi G
+ 2\theta m t_1 \phi^2 |S|^2
- 2(p-1+K)\sqrt s |\nabla\phi| G
\\
& \ \ - \frac{4|p-2|}{\ve} |S| |\nabla\phi| G
- \frac{C_2 q^2}{R^2} G .
\end{align*}
By Young's inequality, using $|\nabla\phi|^2 \le \frac{Cq^2}{R^2} \phi$ from \eqref{bounds} to simplify the first line, we find
\begin{align*}
& 2(p-1+K)\sqrt s\,|\nabla\phi|\, G \;\le\; c_{\alpha,\theta}\, s\,\phi\, G + \frac{C_3(n,p)\, q^2}{(\alpha-1) R^2}\, G,
\\
& \frac{4|p-2|}{\ve}\,|S|\,|\nabla\phi|\,G \;\le\; 2\theta m t_1 \phi^2 |S|^2 + \frac{2(p-2)^2}{\theta m t_1 \ve^2}\,\frac{|\nabla\phi|^2}{\phi^2}\, G^2.
\end{align*}

In Case A below, we use the second bound in \eqref{bounds} to estimate the ratio in the second inequality:
\[
\frac{|\nabla\phi|^2}{\phi^2} \le \frac{C q^2}{R^2\phi^{2/q}}.
\]
\smallskip
\noindent\emph{Case A:} Assume that
\[ 
\phi(x_1)^{2/q} \ge b_R := \dfrac{4C(p-2)^2  q^2 \alpha^2 a_\theta}{\theta\, m\, \ve^2\, R^2}.
\]
Then, by the second bound in \eqref{bounds}, the coefficient of $G^2$ produced by the second Young inequality is at most
\[
\frac{2(p-2)^2}{\theta m t_1 \ve^2}\cdot \frac{Cq^2}{R^2\,\phi(x_1)^{2/q}}
\;\le\; \frac{2(p-2)^2\, C q^2}{\theta m t_1 \ve^2 R^2}\cdot \frac{\theta m \ve^2 R^2}{4C(p-2)^2 q^2 \alpha^2 a_\theta}
\;=\; \frac{1}{2\alpha^2 a_\theta t_1}
\]
(the factor $t_1$ cancels), and
\[
0 \;\ge\; \frac{G^2}{2\alpha^2 a_\theta t_1} - \frac{C_4(n,p)\,q^2}{(\alpha-1)R^2}\, G
\qquad\Longrightarrow\qquad
\mathfrak M \;\le\; \frac{C_5(n,p)\, q^2\, T}{(\alpha-1)\, R^2}\, .
\]

\smallskip
\noindent\emph{Case B:} Suppose instead that
\[
\phi(x_1)^{2/q} < b_R.
\]
Then, by \eqref{G-consequence},
\[
\mathfrak M = \phi(x_1)\, F_{\alpha,\theta}(x_1,t_1)  \le  b_R^{q/2} C_F (1+2R)^N
 \le  C(n,p,q,\theta,\ve,C_F)  R^{N-q} .
\]

\medskip

\noindent \underline{Step 4:} \emph{conclusion.} Since $\phi \equiv 1$ on $\overline B_R$, we have $F_{\alpha,\theta} = G \le \mathfrak M$ on $B_R\times(0,T]$, and, as $q \ge N+1$, in either case
\[
\sup_{B_R\times(0,T]} F_{\alpha,\theta}
\;\le\; \max\Big\{ \frac{C_5\, q^2\, T}{(\alpha-1) R^2}\, ,\; C\, R^{\,N-q} \Big\}
\;\underset{R\to\infty}{\longrightarrow}\; 0,
\]
with $\ve, \alpha, \theta, q$ fixed. Hence $F_{\alpha,\theta} \le 0$ on $\bM\times(0,T]$, and letting $\alpha \downarrow 1$, $\theta \downarrow 0$ gives $F_\ve \le 0$.

\end{proof}

\begin{cor}\label{C:noncompact-limit}
Let $\bM$ be as in Theorem \ref{T:noncompact}, and let $\ve_j \downarrow 0$ and $v^{\ve_j}$ be smooth solutions of \eqref{Eeps} (with $\ve = \ve_j$) on $\bM\times(0,T]$, each satisfying \Hgrow (with constants $C_F, N$ that may depend on $j$). If $v^{\ve_j} \to v$ locally uniformly on $\bM\times(0,T)$, then
\[
v_t \;\ge\; (p-1)|\nabla v|^2 - \frac{n+p-2}{2(p-1)\,t}
\]
in the viscosity sense on $\bM\times(0,T)$; in particular the sharp Li-Yau inequality holds for $u = e^v$ at every point at which $v$ is differentiable.
\end{cor}

\begin{proof}
By Theorem \ref{T:noncompact} and \eqref{Feps-dominates}, each $v^{\ve_j}$ is a classical, hence viscosity, supersolution of the Hamilton-Jacobi equation $\phi_t - (p-1)|\nabla\phi|^2 + \frac at = 0$; the conclusion follows from the stability of viscosity supersolutions under locally uniform convergence, as in the proof of Theorem \ref{T:LY-complete}.

\end{proof}

\begin{rmrk}\label{R:obstruction}
As in Section \ref{S:critical}, Theorem \ref{T:noncompact} makes no nondegeneracy assumption. Noncompactness enters through Hypotheses \ref{H:cut} and \ref{H:grow}, with the latter used only in Case B. This qualitative growth condition follows from the conclusion with $N=0$ and is satisfied by the extremal profile \eqref{model} with $N=0$.

For a prescribed positive viscosity solution $u$ of \eqref{0}, an unconditional result requires approximating flows \eqref{Eeps} satisfying \Hgrow and an identification of their limit with $\log u$. The main difficulty is the latter. For fixed $\ve$, equation \eqref{Eeps} is uniformly parabolic; in $\Rn$, the classical Tychonoff-type existence and uniqueness theory provides the approximations under a growth bound of the form $e^{A|x|^2}$. A comparison principle for the degenerate limit equation \eqref{0} in this growth class is still needed. For $p\neq2$, degeneracy occurs on the critical set, and the available Tychonoff-type comparison theorem \cite[Theorem 4.2]{BG} is restricted to bounded solutions; this forces $\log u$ to be bounded and yields \Hgrow with $N=0$. Extending comparison to solutions of $e^{A|x|^2}$-type growth would make Theorem \ref{T:noncompact} unconditional in $\Rn$. We intend to return to this question and to the noncompact case with $\Ric\ge-\kappa$.

The hypothesis \Hcut enters through the term $F_{\alpha,\theta}L_\ve\phi$, because $L_\ve$ is not in divergence form. The work of Wang-Zhang \cite{WZ} on $p$-harmonic functions suggests that integral methods applied to the regularized flows \eqref{Eeps} may remove this condition under $\Ric\ge0$ alone.
\end{rmrk}

\begin{rmrk}[on the cutoff $\psi^q$]\label{R:psiq}
The choice $\phi=\psi^q$ in Theorem \ref{T:noncompact} serves two purposes. The usual quadratic cutoff gives $|\nabla\phi|^2/\phi \le C/R^2$; a larger power also controls the ratio $|\nabla\phi|^2/\phi^2$ where $\phi$ is not too small, while \Hgrow handles the remaining region.
\begin{itemize}
\item[(i)] Since $\nabla\phi$ vanishes to order $q-1$ on $\{\phi = 0\}$, one has the estimate
\[
\frac{|\nabla\phi|^2}{\phi^2} \;\le\; \frac{Cq^2}{R^2\,\phi^{2/q}}
\]
for the ratio which, in the maximum principle argument, multiplies $G^2$. Consequently the quadratic error term $G^2|\nabla\phi|^2/\phi^2$ produced by the Hessian-dependent part of the drift is dominated by the coercive term $G^2/(\alpha^2a_\theta t)$ at every point where $\phi^{2/q} \ge b_R$, with $b_R$ of order $R^{-2}$.
\item[(ii)] At the remaining points, $\phi < b_R^{q/2}$, which is of order $R^{-q}$; combined with the bound $F_{\alpha,\theta} \le C_F(1+r)^N$ of \Hgrow this gives $G \le CR^{\,N-q}$. Hence a single dichotomy suffices for an arbitrary polynomial growth exponent $N$, subject only to the choice $q > N$.
\end{itemize}

The error in case (i) is of order $q^2$. Taking $q>N$ gives the required decay in case (ii).
\end{rmrk}

\section{The equality case and some open problems}\label{S:final}

By Proposition \ref{P:sharp}, the constant $a = \frac{n+p-2}{2(p-1)}$ in
\eqref{LY} cannot be improved, and the extremal profile is the function
\eqref{model}, which saturates \eqref{LY} identically by \eqref{eq-model}.
It is natural to ask whether the equality case is rigid.

\begin{prob}\label{Prob:rigidity}
Suppose that equality holds in \eqref{LY} at some point $(x_0,t_0)$ (or on
an open set). Is it true that $\bM$ is flat and that $u$ coincides, modulo
translations in $x$ and dilations, with the explicit solution
\eqref{model}?
\end{prob}

When $p = 2$ this is known: tracing back the case of equality in the
Bochner inequality gives $\Ric(\nabla v,\nabla v) \equiv 0$ together with
\[
\nabla_i \nabla_j v = - \frac{1}{2t}\, g_{ij},
\]
and commuting covariant derivatives in the latter identity forces the full
curvature tensor to vanish, so that $\bM$ is flat and, by a rigidity
argument \`a la Weinberger, $v = \log u$ must be quadratic in $x$;
exponentiating, one recovers the Gaussian. For related arguments see
\cite{Ha} and \cite{Ni}. For $p \not= 2$ the same strategy applies
formally at points where $\nabla v \not= 0$ (equality in Lemma
\ref{L:newton-ext} determines the Hessian of $v$), but a complete
proof would also require propagating equality through the critical set.
This remains open.

A second problem is to extend Theorem \ref{T:noncompact} and
Corollary \ref{C:noncompact-limit} to a prescribed positive viscosity
solution of \eqref{0} on a complete noncompact manifold with $\Ric \ge 0$.
As explained in Remark \ref{R:obstruction}, this requires a comparison
principle for the degenerate limit equation \eqref{0} in a class of
$e^{A|x|^2}$-type growth. The approximating flows are uniformly parabolic
for fixed $\ve$ and, in $\Rn$, are covered by the classical Tychonoff-type
theory. For bounded solutions, the required comparison principle is
\cite[Theorem 4.2]{BG}. Other questions concern removing \Hcut under
$\Ric \ge 0$, perhaps through the integral methods of \cite{WZ} discussed
in Remark \ref{R:obstruction}, and treating the noncompact case with
$\Ric \ge -\kappa$.

\section*{Acknowledgments}

We are grateful to Professor Yoshikazu Giga for correspondence concerning
viscosity solutions on Riemannian manifolds, which brought \cite{Gi} and
\cite{AJM} to our attention and, through the latter, led us to
\cite{AC, LW, LTW}.

\section*{Declarations}

\noindent\textbf{Funding.} The authors did not receive support from any
organization for the submitted work.

\smallskip
\noindent\textbf{Competing interests.} The authors have no competing
interests to declare that are relevant to the content of this article.

\smallskip
\noindent\textbf{Data availability.} No datasets were generated or
analysed during the current study.

\smallskip
\noindent\textbf{AI.} Open AI, ChatGpt was used to audit, revise and reorganize the manuscript in its final form.

\end{document}